\documentclass[10pt, reqno]{amsart}
\usepackage{amsfonts,latexsym,enumerate}
\usepackage{amsmath}
\usepackage{amscd}
\usepackage{float,amsmath,amssymb,mathrsfs,bm,multirow,graphics}
\usepackage[dvips]{graphicx}
\usepackage[percent]{overpic}
\usepackage[pdftex]{color}
\usepackage{amsaddr}
\usepackage[numbers,sort&compress]{natbib}

\newcommand{\nequation}{\setcounter{equation}{0}}
\renewcommand{\theequation}{\mbox{\arabic{section}.\arabic{equation}}}
\newcommand{\R}{{\Bbb R}}

\newcommand{\C}{{\Bbb C}}

\newcommand{\proofbegin}{\noindent{\it Proof.\quad}}

\newcommand{\proofend}{\hfill$\Box$\bigskip}
\newcommand{\proofendcontinue}{\hfill \raisebox{.8mm}[0cm][0cm]{$\bigtriangledown$}\bigskip}

\newcommand{\re}{\text{\upshape Re\,}}

\newcommand{\im}{\text{\upshape Im\,}}

\newcommand{\ntlim}{\lim^\angle}

\def\Xint#1{\mathchoice
{\XXint\displaystyle\textstyle{#1}}%
{\XXint\textstyle\scriptstyle{#1}}%
{\XXint\scriptstyle\scriptscriptstyle{#1}}%
{\XXint\scriptscriptstyle\scriptscriptstyle{#1}}%
\!\int}
\def\XXint#1#2#3{{\setbox0=\hbox{$#1{#2#3}{\int}$}
\vcenter{\hbox{$#2#3$}}\kern-.5\wd0}}

\def\dashint{\Xint-}

\newtheorem{theorem}{Theorem}[section]
\newtheorem{claim}{Claim}

\newtheorem{corollary}[theorem]{Corollary}
\newtheorem{lemma}[theorem]{Lemma}

\newtheorem{remark}[theorem]{Remark}

\newtheorem{figuretext}{Figure}

\input epsf
\title[The nonlinear steepest descent method]
{The nonlinear steepest descent method: \\ Asymptotics for initial-boundary value problems}

\author{Jonatan Lenells}
\address{Department of Mathematics, KTH Royal Institute of Technology, \\ 100 44 Stockholm, Sweden.}
\email{jlenells@kth.se}

\begin{document}
\begin{abstract} 
\noindent
We consider the rigorous derivation of asymptotic formulas for initial-boundary value problems using the nonlinear steepest descent method. We give detailed derivations of the asymptotics in the similarity and self-similar sectors for the mKdV equation in the quarter plane. 
Precise and uniform error estimates are presented in detail.  
\end{abstract}

\maketitle



\setcounter{tocdepth}{1}
\tableofcontents

\section{Introduction}\nequation
The nonlinear steepest descent method is an approach for determining the asymptotic behavior of solutions of matrix Riemann-Hilbert (RH) problems that depend on a large parameter. By deforming the contour in such a way that the jump is small everywhere except near a finite number of critical points, detailed asymptotic formulas can be obtained by summing up the contributions from the individual critical points. As the inverse scattering transform formalism expresses the solution of a nonlinear integrable PDE in terms of the solution of a RH problem, the nonlinear steepest descent method is particularly well-suited for determining asymptotic formulas for integrable equations. In fact, the method was first developed by Deift and Zhou in \cite{DZ1993}, where the long time behavior of the solution of the modified Korteweg-de Vries (mKdV) equation on the line was established. 
Given initial data $u_0(x)$, $x \in \R$, in the Schwartz class, they showed that the asymptotics of the associated solution $u(x,t)$ can be determined by performing a triangular factorization of the jump matrix followed by a contour deformation. The latter step could be completed thanks to the introduction of an ingenious analytic approximation of the jump matrix.

In this paper, we consider the derivation of asymptotic formulas for initial-{\it{}boundary} value (IBV) problems. Our goal is to provide a rigorous and yet accessible treatment. We prove two nonlinear steepest descent theorems suitable for finding the asymptotics in the similarity and the self-similar sectors respectively. We then apply these theorems to derive uniformly valid asymptotic formulas.
For definiteness, we consider the mKdV equation
\begin{align}\label{mkdv}
u_t + 6 u^2u_x - u_{xxx} = 0, 
\end{align}
posed in the quarter plane $\{x \geq 0, t \geq 0\}$. 

We first consider the asymptotic behavior in the similarity sector defined by
\begin{align}\label{similaritysector}
t > 1, \qquad x^3/t \to \infty, \qquad x \leq N t, \qquad \text{$N$ constant}.
\end{align}
The condition $x^3/t \to \infty$ defines the left boundary of the sector in the $(x,t)$-plane while the inequality $x \leq Nt$ defines its right boundary.
The solution $u(x,t)$ can be expressed \cite{BFS2004} in terms of the solution $M(x,t,k)$ of a RH problem whose jump matrix involves two spectral functions $r(k)$  and $h(k)$ which are defined in terms of the initial data $u_0(x) = u(x,0)$, $x \geq 0$, and of the boundary values
\begin{align}\label{boundaryvalues}
g_0(t) = u(0,t), \quad g_1(t) = u_x(0,t), \quad g_2(t) = u_{xx}(0,t), \qquad t \geq 0.
\end{align}
In addition to a jump across the real axis, which is present also in the case of the initial value problem, the RH problem for the IBV problem also has jumps across the two lines $\R e^{\frac{i\pi}{3}}$ and $\R e^{\frac{2i\pi}{3}}$, see Figure \ref{Djsreversed.pdf}. The jumps across these two lines involve the spectral function $h(k)$, whereas the jump across $\R$ involves the spectral function $r(k)$. There are two critical points located at $\pm k_0$ where $k_0=\sqrt{\frac{x}{12 t}}$.
At a formal level, it is relatively straightforward to obtain the leading order asymptotics of $u(x,t)$ by reading off the contributions from the two critical points. However, the rigorous derivation of a uniform expansion with precise error terms is more involved. We establish such an expansion (see Theorem \ref{mainth1}) by implementing the following seven steps: 
\begin{enumerate}[1.]\setlength\itemindent{25pt}
\item Find an analytic approximation $h_a(t,k)$  of $h(k)$. 
\item Deform the contour to eliminate the part of the jump that involves $h_a(t,k)$. 
\item Conjugate the RH problem to arrive at an appropriate triangular factorization.
\item Find an analytic approximation of the jump matrix along $\R$.
\item Deform the contour so that the jump is small everywhere except near $\pm k_0$. 
\item Apply a nonlinear steepest descent theorem to find the asymptotics of $M(x,t,k)$. 
\item Find the asymptotics of $u(x,t)$.
\end{enumerate}
Steps 3-7 have analogs in the case of the initial value problem, whereas steps 1 and 2 are unique to the half-line problem. 

We also consider the asymptotic behavior of $u(x,t)$ in the self-similar sector $\mathbb{S}$ defined by
\begin{align}\label{sectorIVdef}
\mathbb{S} = \{(x,t) \, | \, t>1 \; \text{and} \, 0 < x < N t^{1/3}\}, \qquad \text{$N$ constant}.
\end{align}
In this case, the two critical points $\pm k_0$ merge at the origin as $t \to \infty$. Through a series of seven steps, conceptually similar to the seven steps listed above, we derive (see Theorem \ref{mainth2}) an asymptotic formula for the solution. 

Our derivations require that $r(k)$ and $h(k)$ possess a certain amount of regularity. Since regularity of the spectral functions corresponds to decay of the initial and boundary values, this means that the asymptotic formulas only are valid provided that the boundary values in (\ref{boundaryvalues}) have some decay as $t \to \infty$ (see Theorem \ref{existenceth} and Remark \ref{decayremark} for details).
We will not consider the interesting, but difficult, question of deriving asymptotic formulas for problems whose boundary values do not vanish for large $t$.

Other works studying asymptotics in the context of IBV problems for nonlinear integrable PDEs include \cite{BFS2004, BS2009, FI1994, FI1996, FIS2005}. It was shown in \cite{BFS2004} that the solution of the mKdV equation (\ref{mkdv}) on the half-line is $O(t^{-1/2})$ in the asymptotic sector $0 < c_1 < x/t < c_2$. Our asymptotic formula in the similarity region is stronger than this result, because it provides an exact expression for the coefficient of $t^{-1/2}$ in the asymptotic expansion.  In  \cite{FI1996, FIS2005} the solution of the defocusing nonlinear Schr\"odinger (NLS) equation on the half-line was shown to be $O(t^{-1/2})$ for $0 < c_1 < x/t < c_2$, and, for the focusing NLS, formulas for the asymptotic solitons were presented. Similar results were obtained for the KdV equation on the half-line in \cite{FI1994}. The asymptotic behavior of the solution of the Camassa-Holm equation on the half-line was analyzed in \cite{BS2009}.

\subsection{Main results}
The paper presents four main theorems of two different types:

\begin{itemize}
\item[$(i)$] Theorem \ref{steepestdescentth} and Theorem \ref{steepestdescentthIV} are {\it nonlinear steepest descent theorems}. They are of a more general character and are suitable for determining asymptotics of two different classes of RH problems depending on a large parameter.
Although they are formulated with the quarter plane problem for the defocusing mKdV equation (\ref{mkdv}) in mind, they could also be relevant for other initial-boundary value problems (perhaps after appropriate adjustments for different symmetries or dispersion relations). Indeed, the goal has been to formulate these theorems assuming only those properties of the RH problem associated with (\ref{mkdv}) which are relevant for the application of the nonlinear steepest descent method.

\item[$(ii)$] Theorem \ref{mainth1} and Theorem \ref{mainth2} are {\it asymptotic theorems specific to the quarter plane problem for equation (\ref{mkdv})}. Theorem \ref{mainth1} establishes the asymptotics of the solution in the similarity sector and is an application of Theorem \ref{steepestdescentth} to the RH problem associated with (\ref{mkdv}). Similarly, Theorem \ref{mainth2} (see also Corollary \ref{selfsimilarcorollary})  establishes the asymptotics of the solution in the self-similar sector and is an application of Theorem \ref{steepestdescentthIV} to the same RH problem. 
\end{itemize}

We hope that this division of the derivation into two steps will make it easier when applying the approach to other initial-boundary value problems in the future. We also hope that it clarifies the logical structure of the proofs by isolating those properties 
of the RH problem that are essential for the nonlinear steepest descent arguments.

\subsection{Effect of the boundary}
The asymptotic formulas for the half-line problem obtained in Theorems  \ref{mainth1} and \ref{mainth2} are similar to the analogous formulas for the pure initial value problem. In fact, in the similarity sector, the asymptotic formulas on the line and on the half-line have the exact same functional form. The only difference is that for the half-line problem, the definition of the spectral function $r(k)$, which enters the asymptotic formulas, in addition to the initial data, also involves the boundary values  (see Remark \ref{boundaryremark}). In other words, the only effect of the boundary is to modify $r(k)$. Intuitively, we can understand this as follows: Since $x$ grows faster than $t^{1/3}$ in the similarity region, the distance to the boundary eventually gets so big that what happens at the boundary has a small effect on the solution (recall that we assume that the boundary values decay as $t \to \infty$). This means that, from the point of view of the leading order asymptotics in the simlarity sector, the boundary values and the initial data play similar roles. It is therefore not surprising that they enter the asymptotic formula in similar ways. 

In the self-similar sector $\mathbb{S}$, the effect of the boundary is much greater. For the problem on the line, the asymptotics of  the solution $u(x,t)$ in this region is given by
\begin{align}
u(x,t) = \frac{u^P\big(\frac{-x}{(3t)^{1/3}}; s, 0, -s\big)}{(3t)^{\frac{1}{3}}} 
+ O\bigl(t^{-\frac{2}{3}}\bigr),\qquad
t \to \infty, \quad (x,t) \in \mathbb{S},
\end{align}
where $u^P$ denotes a particular solution of the Painlev\'e II equation specified by a complex parameter $s \in i\R$, $|s| < 1$ (see Appendix \ref{painleveapp}). 
For the half-line problem, the asymptotic analysis of the relevant RH problem leads to the same asymptotic formula, except that the value of the parameter $s$ now also depends on the boundary values in addition to the initial data. So far, the situation is analogous to what happens in the similarity region. However, if we consider a quarter plane solution of (\ref{mkdv}) with sufficient smoothness and spatial decay, then the associated spectral functions are not independent, but are related by the so-called global relation. It turns out that, since we are assuming that  the boundary values decay as $t \to \infty$, the global relation forces the parameter $s$ to vanish. The solution $u^P$ corresponding to $s = 0$ vanishes identically, so this means that, for the half-line problem, we have $u(x,t) = O(t^{-\frac{2}{3}})$ as $t \to \infty$ in the self-similar region. This observation can be interpreted as follows: In the self-similar sector, since $x < N t^{1/3}$, the boundary is close enough to have a noticeable effect on the solution. Since we assume that the boundary values decay to zero as $t \to \infty$, the solution gets drawn down to zero faster in the presence of the boundary than in its absence.

\subsection{Organisation of the paper}
In Section \ref{mkdvsec}, we review the construction of solutions of (\ref{mkdv}) in the quarter plane via RH techniques. In Sections \ref{steepsec} and \ref{steepsec2}, we prove the nonlinear steepest descent results Theorem \ref{steepestdescentth} and Theorem \ref{steepestdescentthIV}, respectively. 
In Sections \ref{similaritysec} and \ref{selfsimilarsec}, we apply these theorems to find the asymptotics of the quarter plane solution of (\ref{mkdv}) in the similarity and self-similar sectors, respectively.
A few facts related to the RH problem associated with the Painlev\'e II equation are collected in the appendix. These facts are used in Section \ref{painlevesec} to establish the solution of a model RH problem which is essential for the proof of Theorem \ref{steepestdescentthIV}.

\subsection{Open problems}
We end the introduction by listing three interesting, largely open, problems:
\begin{enumerate}[1.]
\item In this paper we determine the asymptotic behavior of $u(x,t)$ provided that all the boundary values $\{\partial_x^ju(0, t)\}_0^2$ are known. However, for a well-posed problem, only a subset of the initial and boundary values can be independently prescribed. If all boundary values are not known, our asymptotic formulas (see Theorem \ref{mainth1} and Theorem \ref{mainth2}) still provide some information on the solution, but since the function $r(k)$ is unknown, the exact form of the asymptotics remains undetermined. It is an outstanding open problem to derive asymptotic formulas which involve only the prescribed data.

\item The asymptotic formulas of Theorem \ref{mainth1} and Theorem \ref{mainth2} only apply provided that the boundary values $g_j(t) = \partial_x^ju(0,t)$, $j = 0, 1, 2$, possess a certain amount of decay as $t \to \infty$. For example, in view of Theorem \ref{existenceth}, the following decay assumption is sufficient:
$$(1+t)^{11}g_j^{(i)}(t) \in L^1([0,\infty)), \qquad  j = 0, 1, 2, \quad i = 0,1, \dots, 4.$$
However, as already mentioned, only a subset of the boundary values can be prescribed for a well-posed problem. Thus, even though the given boundary data decay as $t \to \infty$, it is not immediate that all the boundary values decay as $t \to \infty$. Is it true that decaying data leads to decaying boundary values in general? The investigation of this question involves an analysis of the generalized Dirichlet to Neumann map. Important progress in this direction has been made by Antonopoulou and Kamvissis for the NLS equation on the half-line. For the defocusing NLS equation with vanishing initial data, they have shown that if the Dirichlet data decay as $t \to \infty$, then so does the Neumann value \cite{AK2015}.

\item Another challenging but very important problem consists of deriving asymptotic formulas for solutions whose boundary values do not decay as $t \to \infty$. In this case, the definition of the spectral functions $r(k)$  and $h(k)$ has to be modified. If the asymptotic behavior of the boundary values as $t \to \infty$ is known, substitutes for $r(k)$  and $h(k)$ can be defined by subtracting off the known asymptotics in the relevant integral equation. Of particular interest in this regard are problems with (asymptotically) time-periodic data. For the NLS equation with periodic boundary conditions, some pioneering results in this direction can be found in \cite{BIK2009, BK2003, BK2007, BKS2009}.
\end{enumerate}

\section{Quarter plane solutions}\nequation\label{mkdvsec}
The mKdV equation (\ref{mkdv}) admits the Lax pair
\begin{align}\label{mulax}
\begin{cases}
\mu_x - ik[\sigma_3, \mu] = \mathsf{U} \mu, \\
\mu_t + 4ik^3[\sigma_3, \mu] = \mathsf{V} \mu,
\end{cases}
\end{align}
where $\mu(x,t,k)$ is a $2\times 2$-matrix valued eigenfunction, $k\in \C$ is the spectral parameter, and 
\begin{align*}
& \sigma_3 = \begin{pmatrix} 1 & 0 \\ 0 & -1 \end{pmatrix}, \qquad
\mathsf{U}(x,t) = \begin{pmatrix} 0 & u \\ u & 0 \end{pmatrix}, 
	\\
& \mathsf{V}(x,t,k) = \begin{pmatrix} -2i u^2 k & -4uk^2 + 2iu_x k - 2 u^3 + u_{xx} \\
-4uk^2 - 2iu_x k - 2 u^3 + u_{xx} & 2i u^2 k \end{pmatrix}.
\end{align*}
We define the spectral functions $\{a(k), b(k), A(k), B(k)\}$ by
\begin{align}
\mathsf{X}(0,k) = \begin{pmatrix} 
\overline{a(\bar{k})} 	&	b(k)	\\
\overline{b(\bar{k})}	&	a(k)
\end{pmatrix},	\qquad 
\mathsf{T}(0,k) = \begin{pmatrix} 
\overline{A(\bar{k})} 	&	B(k)	\\
\overline{B(\bar{k})}	&	A(k)
\end{pmatrix},
\end{align}
where $\mathsf{X}(x,k)$ and $\mathsf{T}(t,k)$ are the solutions of the Volterra integral equations
\begin{align*}
 & \mathsf{X}(x,k) = I + \int_{\infty}^x e^{-i k(x'-x)\hat{\sigma}_3} \mathsf{U}(x',0) \mathsf{X}(x',k) dx',
	\\
& \mathsf{T}(t,k) = I + \int_{\infty}^t e^{4i k^3(t'-t)\hat{\sigma}_3} \mathsf{V}(0,t', k) \mathsf{T}(t',k) dt',
 \end{align*}
and $e^{\hat{\sigma}_3}$ acts on a $2 \times 2$ matrix $M$ by $e^{\hat{\sigma}_3}M = e^{\sigma_3} M e^{-\sigma_3}$. The open domains $\{D_j\}_1^4$ of the complex $k$-plane are defined by (see Figure \ref{Djsreversed.pdf})
\begin{align}\nonumber
D_1 = \{\im k < 0\} \cap \{\im k^3 > 0\},  \qquad
D_2 = \{\im k < 0\} \cap \{\im k^3 < 0\}, 
	\\ \nonumber
D_3 = \{\im k > 0\} \cap \{\im k^3 > 0\},  \qquad
D_4 = \{\im k > 0\} \cap \{\im k^3 < 0\}.
\end{align}
We let $\Sigma$ denote the contour separating the $D_j$'s oriented as in Figure \ref{Djsreversed.pdf}.
\begin{figure}
\begin{center}
\bigskip \bigskip
\begin{overpic}[width=.45\textwidth]{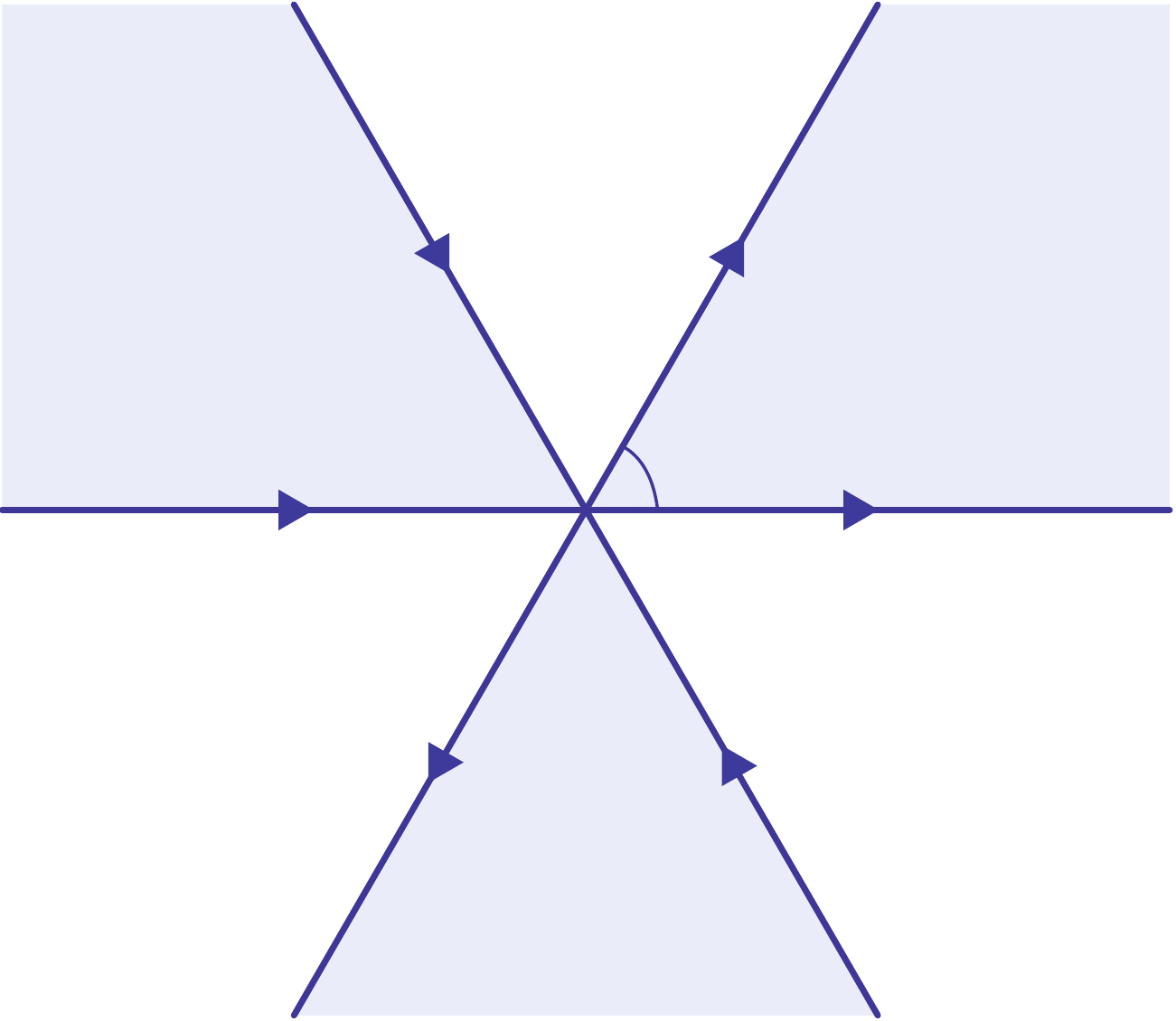}
      \put(77,57){$D_3$}
      \put(46,70){$D_4$}
      \put(15,57){$D_3$}
      \put(15,25){$D_2$}
      \put(46,12){$D_1$}
      \put(77,25){$D_2$}
      \put(57,47){$\pi/3$}
      \put(102,41.5){$\Sigma$}
      \end{overpic}
     \begin{figuretext}\label{Djsreversed.pdf}
       The contour $\Sigma$ and the domains $\{D_j\}_1^4$ in the complex $k$-plane.
     \end{figuretext}
     \end{center}
\end{figure}

We define $h(k)$ and $r(k)$ by
\begin{subequations}\label{hrdef}
\begin{align}\label{hdef}
& h(k) = -\frac{ \overline{B(\bar{k})}}{a(k) d(k)}, \qquad k \in \bar{D}_2,
	\\ \label{rdef}
& r(k) = \frac{\overline{b(\bar{k})}}{a(k)} + h(k), \qquad k \in \R,
\end{align}
\end{subequations}
where $d(k) = a(k)\overline{A(\bar{k})} -  b(k) \overline{B(\bar{k})}$ and $\bar{D}_2$ denotes the closure in $\C$ of the set $D_2$.
It was shown in \cite{Lmkdvrigorous} that if the initial data $u_0(x) = u(x,0)$ and the boundary values $g_j(t) = \partial_x^j u(0,t)$, $j = 0,1,2$, satisfy
\begin{align}\label{ugjassump}
\begin{cases}
u_0 \in C^{m+1}([0,\infty)), &
	\\
(1+x)^{n}u_0^{(i)}(x) \in L^1([0,\infty)), & i = 0,1, \dots, m+1,
	\\
g_j \in C^{[\frac{m+5 - j}{3}]}([0,\infty)), & j = 0,1,2,
	\\
(1+t)^{n}g_j^{(i)}(t) \in L^1([0,\infty)), & j = 0, 1, 2, \quad i = 0,1, \dots, [\frac{m + 5 - j}{3}],
\end{cases}
\end{align}
for some integers $n,m \geq 1$, and the spectral functions satisfy the so-called global relation
\begin{align}\label{GR}
& A(k)b(k) - B(k)a(k) = 0, \qquad k \in \bar{D}_1,
\end{align}
then $r(k)$ and $h(k)$ have the following properties:
\begin{itemize}
\item $r \in C^n(\R)$.

\item $h(k)$ is analytic in $D_2$ and $h^{(j)}(k)$ has a continuous extension to $\bar{D}_2$ for each  $j = 0, 1, \dots, n$.

\item There exist complex constants $\{h_i\}_1^m$ such that, for each $j = 0,1, \dots, n$,
\begin{align}\label{rjkexpansion}
& r^{(j)}(k) = O(k^{-m-1 + 2j}), \qquad |k| \to \infty, \quad k \in \R, 
	\\ \label{hjkexpansion}
& h^{(j)}(k) = \frac{d^j}{dk^j}\bigg(\frac{h_1}{k} + \cdots + \frac{h_m}{k^m}\bigg) + O(k^{-m-1 + 2j}), \qquad k \to \infty, \quad k \in \bar{D}_2.
\end{align}

\item $r(k) = \overline{r(-\bar{k})}$ for $k \in \R$ and $h(k) = \overline{h(-\bar{k})}$ for $k \in \bar{D}_2$.

\item $\sup_{k \in \R} |r(k)| < 1$.
\end{itemize}

The following theorem expresses the quarter plane solution of (\ref{mkdv}) in terms of the solution of a $2 \times 2$-matrix RH problem (see Theorem 7 of \cite{Lmkdvrigorous} for a detailed proof). 

\begin{theorem}\label{existenceth}
Suppose $u_0, g_0, g_1, g_2$ satisfy (\ref{ugjassump}) with $n = 1$ and $m = 4$.
Define the jump matrix $J(x,t,k)$ by 
\begin{align}\label{Jdef}
&J(x,t,k) = \begin{cases} 
 \begin{pmatrix} 1 & 0 \\ h(k) e^{-2ikx + 8ik^3t} & 1 \end{pmatrix}, & k \in \partial D_1,
	\\
 \begin{pmatrix} 1 & - \overline{r(\bar{k})} e^{2ikx - 8ik^3t} \\
 r(k)e^{-2ikx + 8ik^3t}& 1 -  |r(k)|^2\end{pmatrix}, & k \in \R,
	\\ 
 \begin{pmatrix} 1 & \overline{h(\bar{k})} e^{2ikx - 8ik^3t} \\ 0 & 1 \end{pmatrix}, & k \in \partial D_4.
\end{cases}
\end{align}
Suppose the homogeneous RH problem determined by $(\Sigma, J(x,t,\cdot))$ has only the trivial solution for each $(x,t) \in [0, \infty) \times [0,\infty)$.
Suppose the spectral functions satisfy the global relation (\ref{GR}).

Then the $L^2$-RH problem\footnote{Here the generalized Smirnoff class $\dot{E}^2(\C\setminus \Sigma)$ consists of functions $f(k)$ analytic in $\C\setminus \Sigma$ with the property that for each component $D_j$ of $\C\setminus \Sigma$ there exist curves $\{C_n\}_1^\infty$ in $D_j$ such that $C_n$ eventually surround each compact subset of $D_j$ and $\sup_{n \geq 1} \|f\|_{L^2(C_n)} < \infty$.}
\begin{align}\label{RHM}
\begin{cases}
M(x, t, \cdot) \in I + \dot{E}^2(\C \setminus \Sigma),\\
M_+(x,t,k) = M_-(x, t, k) J(x, t, k) \quad \text{for a.e.} \ k \in \Sigma,
\end{cases}
\end{align}
has a unique solution for each $(x,t) \in [0,\infty) \times [0, \infty)$. Moreover, the nontangential limit\footnote{The notation $\displaystyle{\ntlim_{k\to \infty}}$ indicates the limit as $k \in \C$ approaches $\infty$  nontangentially with respect to $\Gamma$.}
\begin{align}\label{ulim}
u(x,t) = -2i\ntlim_{k\to\infty} (kM(x,t,k))_{12}
\end{align}
exists for each $(x,t) \in [0,\infty) \times [0, \infty)$ and the function $u(x,t)$ defined by (\ref{ulim}) is a solution of the mKdV equation (\ref{mkdv}) in the quarter plane $\{x \geq 0, t \geq0\}$ such that 
\begin{enumerate}[$(a)$]
\item $u:[0,\infty) \times [0, \infty) \to \R$ is $C^3$ in $x$ and $C^1$ in $t$.

\item $u(0,t) = g_0(t)$, $u_x(0,t) = g_1(t)$, and $u_{xx}(0,t) = g_2(t)$ for $t \geq 0$.

\item $u(x,0) = u_0(x)$ for $x \geq 0$.
\end{enumerate}
\end{theorem}

The long time asymptotics of the quarter plane solution $u(x,t)$  can be determined via a nonlinear steepest descent analysis of the RH problem (\ref{RHM}). The next section contains a nonlinear steepest descent theorem (Theorem \ref{steepestdescentth}) suitable for finding the asymptotics of $u$ in the similarity sector (\ref{similaritysector}).

\begin{remark}\upshape
Equation (\ref{rjkexpansion}) shows that for smooth initial and boundary values, the spectral function $r(k)$ decays quickly as $k \to \infty$. Thus, as in the case of the initial value problem, the jump across $\R$ vanishes quickly as $k \to \infty$. In contrast, regardless of how smooth the initial and boundary values are, the function $h(k)$ in general only decays as $1/k$ as $k \to \infty$, see (\ref{hjkexpansion}). We will overcome this difficulty by making a careful choice of the analytic approximation $h_a$. 
\end{remark}

\section{A nonlinear steepest descent theorem} \label{steepsec}
The goal of this section is to prove a nonlinear steepest descent result (Theorem \ref{steepestdescentth}) suitable for determining asymptotics of a class of RH problems which arise in the study of long-time asymptotics in the similarity sector.
Although the theorem is formulated with equation (\ref{mkdv}) in mind, analogous results can be formulated for other IBV problems. Before stating the theorem, we describe the set-up and make a number of assumptions (see Remark \ref{intuitionremark} for a short discussion of the intuition behind these assumptions).

Let $X$ denote the cross $X = X_1 \cup \cdots \cup X_4 \subset \C$ where the rays
\begin{align} \nonumber
&X_1 = \bigl\{se^{\frac{i\pi}{4}}\, \big| \, 0 \leq s < \infty\bigr\}, && 
X_2 = \bigl\{se^{\frac{3i\pi}{4}}\, \big| \, 0 \leq s < \infty\bigr\},  
	\\ \label{Xdef}
&X_3 = \bigl\{se^{-\frac{3i\pi}{4}}\, \big| \, 0 \leq s < \infty\bigr\}, && 
X_4 = \bigl\{se^{-\frac{i\pi}{4}}\, \big| \, 0 \leq s < \infty\bigr\},
\end{align}
are oriented as in Figure \ref{X.pdf}. 
For $r > 0$, let $X^r = X_1^r \cup \cdots \cup X_4^r$ denote the restriction of $X$ to the open disk of radius $r$ centered at the origin, i.e. $X^r = X \cap \{|z| < r\}$. 

Let $\mathcal{I} \subset \R$ be an interval. Let $\rho, \epsilon, k_0:\mathcal{I} \to (0,\infty)$ be bounded strictly positive functions such that $\epsilon(\zeta) < k_0(\zeta)$ for $\zeta \in \mathcal{I}$. We henceforth drop the $\zeta$ dependence of these functions and write simply $\rho$, $\epsilon$, $k_0$ for $\rho(\zeta)$, $\epsilon(\zeta)$, $k_0(\zeta)$, respectively. 

Let $\Gamma = \Gamma(\zeta)$ be a family of oriented contours parametrized by $\zeta \in \mathcal{I}$ and let 
$$\hat{\Gamma} := \Gamma \cup \{k \in \C \, | \, |k \pm k_0| = \epsilon \}$$ 
denote the union of $\Gamma$ with the circles of radius $\epsilon$ centered at $\pm k_0$ oriented counterclockwise. Assume that, for each $\zeta \in \mathcal{I}$:
\begin{itemize}

\item[($\Gamma$1)] $\Gamma$ and $\hat{\Gamma}$ are Carleson jump contours up to reorientation of a subcontour.

\item[($\Gamma$2)] $\Gamma$ contains the two crosses $\pm k_0 + X^{\epsilon}$ as oriented subcontours.

\item[($\Gamma$3)] $\Gamma$ is invariant as a set under the map $k \mapsto - \bar{k}$. Moreover, the orientation of $\Gamma$ is such that if $k$ traverses $\Gamma$ in the positive direction, then $-\bar{k}$ traverses $\Gamma$ in the negative direction.

\item[($\Gamma$4)] The point $\infty$ in the Riemann sphere can be approached nontangentially with respect to $\Gamma$.
\end{itemize}
Assume that the Cauchy singular operator $\mathcal{S}_{\hat{\Gamma}}$ defined by
\begin{align}\label{Cauchysingulardef}
(\mathcal{S}_{\hat{\Gamma}} h)(z) = \lim_{r \to 0} \frac{1}{\pi i} \int_{\hat{\Gamma} \setminus \{z' \, | \, |z' - z| < r\}} \frac{h(z')}{z' - z} dz',
\end{align}
is uniformly bounded on $L^2(\hat{\Gamma})$, i.e. $\sup_{\zeta \in \mathcal{I}} \|\mathcal{S}_{\hat{\Gamma}}\|_{\mathcal{B}(L^2(\hat{\Gamma}))} < \infty$.

\begin{figure}
\begin{center}
 \begin{overpic}[width=.3\textwidth]{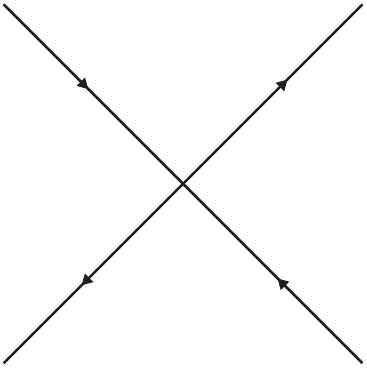}
 \put(66,81){$X_1$}
 \put(24,81){$X_2$}
 \put(24,16){$X_3$}
 \put(66,16){$X_4$}
 \end{overpic}
   \bigskip
   \begin{figuretext}\label{X.pdf}
      The contour $X = X_1 \cup \cdots \cup X_4$.
      \end{figuretext}
   \end{center}
\end{figure}
Consider the following family of $L^2$-RH problems parametrized by the two parameters $\zeta \in \mathcal{I}$ and $t > 0$:
\begin{align}\label{RHm}
\begin{cases} m(\zeta, t, \cdot) \in I + \dot{E}^2(\C \setminus \Gamma), \\
m_+(\zeta, t, k) = m_-(\zeta, t, k) v(\zeta, t, k) \quad \text{for a.e.} \ k \in \Gamma, 
\end{cases} 
\end{align}
where the jump matrix $v(\zeta, t, k)$ satisfies
\begin{align}\label{winL1L2Linf}
& w(\zeta, t,\cdot) := v(\zeta, t,\cdot) - I \in L^1(\Gamma) \cap L^2(\Gamma) \cap L^\infty(\Gamma), \qquad \zeta \in \mathcal{I}, \quad t > 0,
	\\
&\det v(\zeta, t, \cdot) = 1 \;\; \text{a.e. on $\Gamma$}, \qquad \zeta \in \mathcal{I}, \quad t > 0, \quad 
\end{align}
and
\begin{align}\label{vsymm}
  v(\zeta,t,k) = \overline{v(\zeta, t, -\bar{k})}, \qquad \zeta \in \mathcal{I}, \quad t > 0, \quad k \in \Gamma.
\end{align}

Let $\tau := t \rho^2$. Let $\Gamma_X$ be the union of the two crosses $\pm k_0 + X^{\epsilon}$ and let $\Gamma' = \Gamma \setminus \Gamma_X$. Suppose 
\begin{align} \label{wL12infty}
\begin{cases}
 \|w(\zeta,t,\cdot)\|_{L^1(\Gamma')} = O(\epsilon \tau^{-1}), 
	\\ 
 \|w(\zeta,t,\cdot)\|_{L^\infty(\Gamma')} = O(\tau^{-1}), 
\end{cases}\qquad \tau \to \infty, \quad \zeta \in \mathcal{I}, 
\end{align}
uniformly with respect to $\zeta \in \mathcal{I}$. 
Moreover, suppose that near $k_0$ the normalized jump matrix
\begin{align}\label{vjdef}
v_0(\zeta,t,z) = \sigma_3 v\biggl(\zeta, t, k_0 - \frac{\epsilon z}{\rho}\biggr)\sigma_3, \qquad z \in X^{\rho},
\end{align}
has the form
\begin{align}\label{smallcrossjump}
v_0(\zeta, t, z) = \begin{cases}
\begin{pmatrix} 1 & 0	\\
  R_1(\zeta, t, z)z^{-2i\nu(\zeta)} e^{t\phi(\zeta, z)}	& 1 \end{pmatrix}, &  z \in X_1^{\rho}, \\
\begin{pmatrix} 1 & -R_2(\zeta, t, z)z^{2i\nu(\zeta)}e^{-t\phi(\zeta, z)}	\\
0 & 1  \end{pmatrix}, &  z \in X_2^{\rho}, \\
\begin{pmatrix} 1 &0 \\
 -R_3(\zeta, t, z)z^{-2i\nu(\zeta)} e^{t\phi(\zeta, z)}	& 1 \end{pmatrix}, &  z \in X_3^{\rho},
 	\\
 \begin{pmatrix} 1	& R_4(\zeta, t, z)z^{2i\nu(\zeta)}e^{-t\phi(\zeta, z)}	\\
0	& 1 \end{pmatrix}, & z \in X_4^{\rho}, 
\end{cases}
\end{align}
where:
\begin{itemize}
\item The phase $\phi(\zeta, z)$ is a smooth function of $(\zeta, z) \in \mathcal{I} \times \C$ 
such that 
\begin{align}\label{phiassumptions}
\phi(\zeta, 0) \in i\R, \qquad \frac{\partial \phi}{\partial z}(\zeta, 0) = 0, \qquad \frac{\partial^2 \phi}{\partial z^2}(\zeta, 0) = i,\qquad \zeta \in \mathcal{I},
\end{align}
and
\begin{subequations}
\begin{align} \label{rephiestimatea}
 & \re \phi(\zeta,z) \leq -\frac{|z|^2}{4}, &&  z \in X_1^{\rho} \cup X_3^{\rho}, \quad \zeta \in \mathcal{I},
  	\\ \label{rephiestimateb}
 & \re \phi(\zeta,z) \geq \frac{|z|^2}{4}, &&  z \in X_2^{\rho} \cup X_4^{\rho},  \quad \zeta \in \mathcal{I},
  	\\ \label{Phiz3estimate}
 & \biggl|\phi(\zeta, z) - \phi(\zeta,0) - \frac{iz^2}{2}\biggr| \leq C \frac{|z|^3}{\rho}, && z \in X^{\rho}, \quad \zeta \in \mathcal{I},
\end{align}
\end{subequations}
where $C > 0$ is a constant.

\item There exist a function $q:\mathcal{I} \to \C$ with $\sup_{\zeta \in \mathcal{I}} |q(\zeta)| < 1$, and constants $(\alpha, L) \in [\frac{1}{2},1) \times (0,\infty)$ such that the functions $\{R_j(\zeta, t, z)\}_1^4$ satisfy the following inequalities:
\begin{align} \label{Lipschitzconditions}
\begin{cases}
   |R_1(\zeta, t, z) - q(\zeta)| \leq  L  \bigl| \frac{z}{\rho}\bigr|^\alpha e^{\frac{t|z|^2}{6}}, & z \in X_1^{\rho},  \\
 \left|R_2(\zeta, t, z) - \frac{\overline{q(\zeta)}}{1 - |q(\zeta)|^2}\right| \leq  L \bigl| \frac{z}{\rho}\bigr|^\alpha e^{\frac{t|z|^2}{6}}, \qquad& z \in X_2^{\rho}, \vspace{.1cm}\\ 
  \left|R_3(\zeta, t, z) - \frac{q(\zeta)}{1 - |q(\zeta)|^2}\right| \leq  L \bigl| \frac{z}{\rho}\bigr|^\alpha e^{\frac{t|z|^2}{6}}, & z \in X_3^{\rho}, 
  	\\
  |R_4(\zeta, t, z) - \overline{q(\zeta)}| \leq L  \bigl| \frac{z}{\rho}\bigr|^\alpha e^{\frac{t|z|^2}{6}}, & z \in X_4^{\rho}, 
\end{cases} \quad \zeta \in \mathcal{I}, \quad t > 0.
\end{align}
\item The function $\nu(\zeta)$ is defined by $\nu(\zeta) = -\frac{1}{2\pi} \ln(1 -  |q(\zeta)|^2)$.
\end{itemize}

\begin{theorem}[Nonlinear steepest descent]\label{steepestdescentth}
Under the above assumptions, the $L^2$-RH problem (\ref{RHm}) has a unique solution for all sufficiently large $\tau$ and this solution satisfies
\begin{align}\label{limlm12}
\ntlim_{k\to \infty} (k m(\zeta,t,k))_{12}
= -\frac{2i\epsilon \re{\beta(\zeta, t)}}{\sqrt{\tau}}  + O\bigl(\epsilon \tau^{-\frac{1+\alpha}{2}} \bigr), \qquad \tau \to \infty, \quad \zeta \in \mathcal{I},
\end{align}
where the error term is uniform with respect to $\zeta \in \mathcal{I}$ and $\beta(\zeta, t)$ is defined by
\begin{align}\label{betadef}
  \beta(\zeta, t) = \sqrt{\nu(\zeta)} e^{i\left(\frac{\pi}{4} - \arg q(\zeta) + \arg \Gamma(i\nu(\zeta))\right)} e^{-t\phi(\zeta, 0)} t^{-i\nu(\zeta)}. 
\end{align}  
The existence of the nontangential limit in (\ref{limlm12}) for all sufficiently large $\tau$ is part of the conclusion of the theorem.
\end{theorem}
\begin{proof}
We omit the proof since it is similar to the proof of an analogous theorem in \cite{Lnonlinearsteepest}.
\end{proof}

\begin{remark}\upshape\label{intuitionremark}
When applying Theorem \ref{steepestdescentth} to the mKdV equation in Section \ref{similaritysec}, we will identify $\zeta$ with the quotient $\zeta := x/t$ and choose $\mathcal{I} = (0, N]$ with  $N > 0$.
Moreover, the two critical points will be located at $\pm k_0$ with $k_0 := \sqrt{\zeta/12}$ and we will take $\epsilon := k_0/2$ and $\rho := \epsilon \sqrt{48 k_0}$, so that $\tau = t\rho^2 =  \sqrt{x^3/(12 t)}$.
The contours $\Gamma$ and $\hat{\Gamma}$ will be as displayed in Figure \ref{Gammahat.pdf}.
The idea is that the disks of radius $\epsilon$ centered at $\pm k_0$ provide small neighborhoods of the two critical points in the $k$-plane, such that as $t \to \infty$ the jump matrix is close to the identity matrix everywhere except on the two small crosses $ \pm k_0 + X^{\epsilon}$ (cf. assumption (\ref{wL12infty})). 
Near the critical point $k_0$, it is convenient to work with the variable $z = \frac{\rho(k_0 -k)}{\epsilon}$ (i.e. $k = k_0 - \frac{\epsilon z}{\rho}$) which centers the critical point at the origin and maps the cross $k_0 + X^{\epsilon}$ of radius $\epsilon$ in the $k$-plane to the cross $X^\rho$ of radius $\rho$ in the $z$-plane. The scaling factor $\rho$ is chosen so that the jump matrix takes the standard form (\ref{smallcrossjump}) in the $z$-plane.
The variable $\tau$ is introduced so that $\tau \to \infty$ corresponds to the condition $x^3/t \to \infty$ in the definition (\ref{similaritysector}) of the similarity sector.
\end{remark}

\section{Asymptotics in the similarity sector}\label{similaritysec}\nequation
The goal of this section is to prove Theorem \ref{mainth1} which provides an asymptotic formula for the quarter plane solutions of the mKdV equation (\ref{mkdv}) in the similarity sector.
The proof is essentially an application of Theorem \ref{steepestdescentth}.

Let $\zeta = x/t$ and define the variables $k_0 = k_0(\zeta)$ and $\tau = \tau(x,t)$ by
\begin{align}
k_0 = \sqrt{\frac{\zeta}{12}}, \qquad \tau = 12 t k_0^3.
\end{align}

\begin{theorem}[Asymptotics in the similarity sector]\label{mainth1}
Suppose $r(k)$ is a function in $C^{11}(\R)$ such that 
\begin{itemize}
\item $r(k) = \overline{r(-\bar{k})}$ for $k \in \R$.

\item $\sup_{k \in \R} |r(k)| < 1$.

\item $r^{(j)}(k) = O(k^{-4 + 2j})$ as $|k| \to \infty$, $k \in \R$, for $j = 0,1, 2$.
\end{itemize}
Suppose $h(k)$ is an analytic function of $k \in D_2$ such that
\begin{itemize}
\item $h^{(j)}(k)$ has a continuous extension to $\bar{D}_2$ for each $j = 0,1, \dots, 5$, that is, $h\in C^5(\bar{D}_2)$.

\item $h(k) = \overline{h(-\bar{k})}$ for $k \in \bar{D}_2$.

\item There exist complex constants $\{h_i\}_1^3$ such that
$$h^{(j)}(k) = \frac{d^j}{dk^j}\bigg(\frac{h_1}{k} + \cdots + \frac{h_3}{k^3}\bigg) + O(k^{-4 + 2j}), \qquad k \to \infty, \quad k \in \partial D_1, \quad j = 0,1, 2.$$
\end{itemize}
Let $\alpha \in [1/2,1)$ and let $N > 0$. 

Then there exists a $T>0$ such that the $L^2$-RH problem (\ref{RHM}) has a unique solution and the limit in (\ref{ulim}) exists whenever $\tau > T$ and $0 < x < Nt$.
Moreover, the function $u(x,t)$ defined by (\ref{ulim}) satisfies
\begin{align}\label{ufinal}
u(x,t) = -\frac{1}{\sqrt{t k_0}}\bigl[u_a(x,t) + O\bigl(\tau^{-\frac{\alpha}{2}}\bigr)\bigr],\qquad
\tau \to \infty, \quad 0 < x < Nt,
\end{align}
where the error term is uniform with respect to $x$ in the given range and the function $u_a$ is defined by
\begin{align}\label{uadef}
& u_a(x,t) = \sqrt{\frac{\nu(\zeta)}{3}} \cos\big(16tk_0^3 - \nu(\zeta) \ln(192 tk_0^3) + \phi(\zeta)\big)
\end{align}
with
\begin{align}\label{phizetadef}
& \phi(\zeta) = \frac{\pi}{4} + \arg \Gamma(i\nu(\zeta)) - \arg r(k_0) 
+ \frac{1}{\pi}\dashint_{\R} \frac{\psi(\zeta, s)ds}{s - k_0},
	\\ \label{nudef}
& \nu(\zeta) = -\frac{1}{2\pi} \ln(1 - |r(k_0)|^2),
	\\ \label{psidef}
& \psi(\zeta, s) = 
\begin{cases} \ln(1 - |r(s)|^2), & |s| > k_0, 
	\\
\ln(1 - |r(k_0)|^2), & |s| < k_0, 
\end{cases} \quad s \in \R.
\end{align}
\end{theorem}

\begin{remark}[Quarter plane solutions]\label{decayremark}\upshape
Theorem \ref{mainth1} holds for any functions $r(k)$  and $h(k)$ which satisfy the stated assumptions, regardless of whether these functions are defined via (\ref{hrdef}) or not. By assuming that $r(k)$  and $h(k)$ are defined in terms of $u_0(x)$ and $\{g_j(t)\}_0^2$ via (\ref{hrdef}) and by combining Theorem \ref{existenceth}  and Theorem \ref{mainth1}, we can ensure that the function $u(x,t)$ defined in (\ref{ulim}) constitutes a quarter plane solution of (\ref{mkdv}) with initial and boundary values given by $u_0$ and $\{g_j\}_0^2$. More precisely, we have the following corollary: Suppose $u_0, g_0, g_1, g_2$ satisfy (\ref{ugjassump}) with $n = 11$ and $m = 4$, that the associated spectral functions satisfy the global relation (\ref{GR}), and that the $L^2$-RH problem (\ref{RHM}) has a solution for each $(x,t) \in [0,\infty) \times [0,\infty)$. Then the function $u(x,t)$  defined by (\ref{ulim}) is a well-defined classical solution of (\ref{mkdv}) in the quarter plane with initial data $u(x,0) = u_0(x)$  and boundary values $\partial_x^ju(0,t) = g_j(t)$, $j = 0,1,2$. Moreover, the asymptotics of $u(x,t)$ in the similarity sector is given by (\ref{ufinal}). 
\end{remark}

\begin{remark}[Effect of the boundary]\label{boundaryremark}\upshape
According to equations (\ref{uadef})-(\ref{psidef}), the asymptotic wave shape $u_a(x,t)$ associated with a quarter plane solution $u(x,t)$ depends on the initial and boundary values via the function 
$$r(k) = \frac{\overline{b(\bar{k})}}{a(k)} - \frac{ \overline{B(\bar{k})}}{a(k) d(k)},$$ 
where $d(k) = a(k)\overline{A(\bar{k})} -  b(k) \overline{B(\bar{k})}$. The functions $\{a(k), b(k)\}$ are defined in terms of the initial data, whereas the functions $\{A(k), B(k)\}$ are defined in terms of the boundary values. If we considered equation (\ref{mkdv}) on the line (instead of the half-line), the solution would satisfy an asymptotic formula analogous to (\ref{ufinal}) but with $r(k)$ replaced with 
$$r^L(k) = \frac{\overline{b^L(\bar{k})}}{a^L(k)},$$ 
where $\{a^L(k), b^L(k)\}$ are the line analogs of the half-line spectral functions $\{a(k), b(k)\}$. It follows that the effect of the boundary on the leading-order asymptotics in the similarity sector is to add the term $- \overline{B(\bar{k})}/(a(k) d(k))$ to the value of $r(k)$.
\end{remark}

\medskip\noindent
{\it Proof of Theorem \ref{mainth1}.}
Let $N > 1$ be given and let $\mathcal{I}$ denote the interval $\mathcal{I} = (0, N]$. 
The jump matrix $J$ defined in (\ref{Jdef}) involves the exponentials $e^{\pm t\Phi(\zeta,k)}$, where  $\Phi(\zeta, k)$ is defined by
$$\Phi(\zeta, k) = 8ik^3 - 2ik \zeta, \qquad \zeta \in \mathcal{I}, \quad k \in \C.$$
It follows that there are two stationary points located at the points where $\frac{d\Phi}{dk} = 0$, i.e. at $k = \pm k_0$. The real part of $\Phi$ is shown in Figure \ref{rePhi.pdf}. 

\medskip
{\bf Step 1: Introduce an analytic approximation of $h(k)$.}
The first step of the proof consists of decomposing $h$ into an analytic part $h_a$ and a small remainder $h_r$.

\begin{lemma}\label{decompositionlemma}
There exists a decomposition
\begin{align*}
h(k) = h_{a}(t, k) + h_{r}(t, k), \qquad t > 0, \quad k \in \partial D_1,  
\end{align*}
where the functions $h_{a}$ and $h_{r}$ have the following properties:
\begin{enumerate}[$(a)$]
\item For each $t > 0$, $h_{a}(t, k)$ is defined and continuous for $k \in \bar{D}_1$ and analytic for $k \in D_1$.

\item For each $\zeta \in \mathcal{I}$ and each $t > 0$, 
\begin{align*}
|h_{a}(t, k)| \leq 
\frac{C}{1 + |k|} e^{\frac{t}{2} |\re \Phi(\zeta,k )|},  \qquad k \in \bar{D}_1,
\end{align*}
where the constant $C$ is independent of $\zeta, t, k$.

\item The $L^1, L^2$, and $L^\infty$ norms of the function $h_{r}(t, \cdot)$ on $\partial D_1$ are $O(t^{-3/2})$ as $t \to \infty$.

\item The following symmetries are valid:
\begin{align}\label{hsymmetries}
h_{a}(t, k) = \overline{h_{a}(t, -\bar{k})}, \quad
h_{r}(t, k) = \overline{h_{r}(t, -\bar{k})}.
\end{align}
\end{enumerate}
\end{lemma}
\begin{figure}
\begin{center}
\bigskip
  \begin{overpic}[width=.6\textwidth]{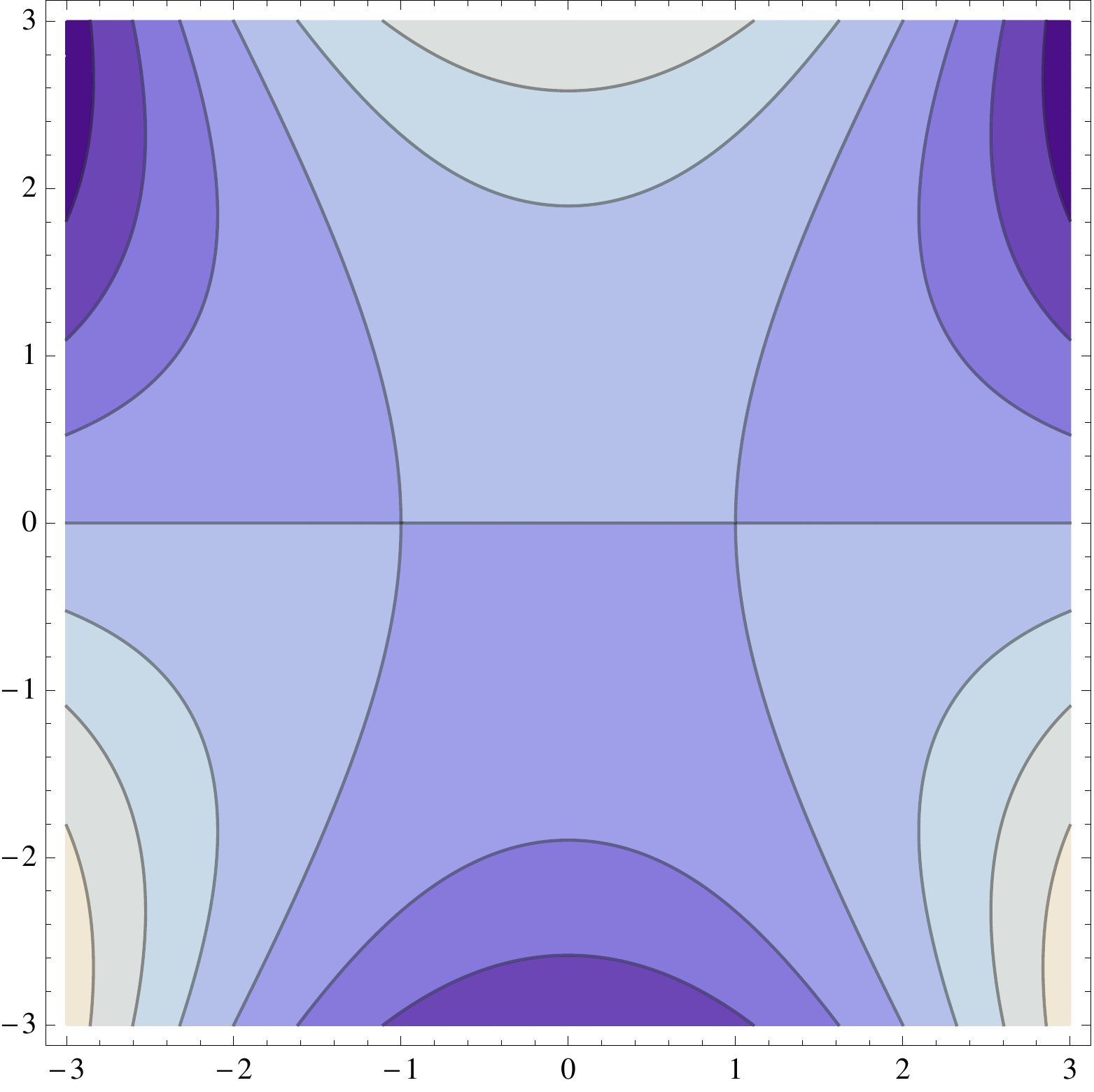}
  \put(48,-4){$\re k$}
   \put(-8,51){$\im k$}
   \put(63,47.5){$k_0$}
   \put(35,47.5){$-k_0$}
   \put(14,42){$\re \Phi > 0$}
   \put(43,67){$\re \Phi > 0$}
   \put(72,42){$\re \Phi > 0$}
   \put(14,58){$\re \Phi < 0$}
   \put(43,35){$\re \Phi < 0$}
   \put(72,58){$\re \Phi < 0$}
     \end{overpic}
   \bigskip
   \begin{figuretext}\label{rePhi.pdf}
      Contour plot of $\re \Phi(\zeta, k)$ for $k_0 = 1$. Light (dark) regions correspond to positive (negative) values of $\re \Phi$. 
            \end{figuretext}
   \end{center}
\end{figure}
\proofbegin
Since $h \in C^5(\bar{D}_2)$, there exist complex constants $\{p_j\}_0^4$ such that
\begin{align*}
& h^{(n)}(k) = \frac{d^n}{dk^n}\bigg(\sum_{j=0}^4 p_j k^j\bigg) + O(k^{5-n}), \qquad k \to 0, \quad k \in \partial D_1, \quad n = 0,1,2.
\end{align*}
In fact, $p_j = \frac{h^{(j)}(0)}{j!}$ for $j = 0, 1, \dots, 4$.
Let
$$f_0(k) = \sum_{j=1}^8 \frac{a_j}{(k - i)^j},$$
where $\{a_j\}_1^8$ are complex constants such that
\begin{align}\label{linearconditions}
f_0(k) = 
\begin{cases}
\sum_{j=0}^4 p_j k^j + O(k^5), & k \to 0,
	\\
\sum_{j=1}^3 h_j k^{-j} + O(k^{-4}), & k \to \infty.
\end{cases}
\end{align}
It is easy to verify that (\ref{linearconditions}) imposes eight linear conditions on the $a_j$'s that are linearly independent; hence the coefficients $a_j$ exist and are unique. The symmetry $h(k) = \overline{h(-\bar{k})}$ implies that $a_j = (-1)^j \bar{a}_j$, $j = 1, \dots, 8$. Hence $f_0(k) = \overline{f_0(-\bar{k})}$. Letting $f = h - f_0$, it follows that:
\begin{enumerate}[$(i)$]
\item $f_0(k)$ is a rational function of $k \in \C$ with no poles in $\bar{D}_1$.

\item $f_0(k)$ coincides with $h(k)$ to order four at $0$ and to order three at $\infty$; more precisely,
\begin{align}\label{hcoincide}
 \frac{\partial^n f}{\partial k^n} (k) =
\begin{cases}
 O(k^{5 - n}), & k \to 0, 
	\\
O(k^{-4+2n}), & k \to \infty, 
 \end{cases}
 \qquad  k \in \partial D_1, \quad n = 0,1,2.
\end{align}

\item $f(k) = \overline{f(-\bar{k})}$ for $k \in \C$.
\end{enumerate}

The decomposition of $h(k)$ can now be derived as follows.
The map $k \mapsto \phi = \phi(k)$ defined by
$$\phi(k) = 8 k^3$$
is a bijection  $\partial D_1 \to \R$, so we may define a function $F:\R \to \C$ by
\begin{align}\label{Fphi}
F(\phi) = (k - i)^2 f(k), \qquad \zeta \in \mathcal{I}, \quad \phi \in \R.
\end{align}
The function $F(\phi)$ is $C^5$ for $\phi \neq 0$ and 
$$F^{(j)}(\phi) = \bigg(\frac{1}{24 k^2} \frac{\partial }{\partial k}\bigg)^j [(k-i)^2f(k)], \qquad \phi \in \R \setminus \{0\}.$$ 
By (\ref{hcoincide}), $F \in C^1(\R)$ and $F''(\phi) = O(|\phi|^{-1/3})$ as $\phi \to 0$, while $F^{(j)}(\phi) = O(|\phi|^{-2/3})$ as $|\phi| \to \infty$ for $j = 0,1,2$.
Hence
\begin{align*}
\bigg\|\frac{d^jF}{d \phi^j}\bigg\|_{L^2(\R)} < \infty, \qquad j = 0,1,2,
\end{align*}
that is, $F$ belongs to the Sobolev space $H^2(\R)$.
It follows that the Fourier transform $\hat{F}(s)$ defined by
$$\hat{F}(s) = \frac{1}{2\pi} \int_{\R} F(\phi) e^{-i\phi s} d\phi,$$
satisfies
\begin{align}\label{FphihatF}
F(\phi) =  \int_{\R} \hat{F}(s) e^{i\phi s} ds,
\end{align}
and
$$\|s^2 \hat{F}(s)\|_{L^2(\R)} < \infty.$$
Equations (\ref{Fphi}) and (\ref{FphihatF}) imply
$$ \frac{1}{(k-i)^2}\int_{\R} \hat{F}(s) e^{8isk^3} ds 
= f(k), \qquad k \in \partial D_1.$$
Writing
$$f(k) = f_a(t, k) + f_r(t, k), \qquad t > 0, \quad k \in \partial D_1,$$
where the functions $f_a$ and $f_r$ are defined by
\begin{align*}
& f_a(t,k) = \frac{1}{(k-i)^2}\int_{-\frac{t}{2}}^\infty \hat{F}(s) e^{8isk^3} ds, \qquad t > 0, \quad k \in \bar{D}_1,  
	\\
& f_r(t,k) = \frac{1}{(k-i)^2}\int_{-\infty}^{-\frac{t}{2}} \hat{F}(s) e^{8isk^3} ds,\qquad t > 0, \quad  k \in \partial D_1,
\end{align*}
we infer that $f_a(t, \cdot)$ is continuous in $\bar{D}_1$ and analytic in $D_1$. 
Furthermore, since $|\re 8ik^3| \leq |\re \Phi(\zeta, k)|$ for all $k \in \bar{D}_1$ and $\zeta \in \mathcal{I}$, we find
\begin{align*}\nonumber
 |f_a(t, k)| 
&\leq \frac{1}{|k-i|^2}\|\hat{F}\|_{L^1(\R)}  \sup_{s \geq -\frac{t}{2}} e^{s \re 8ik^3}
\leq \frac{C}{|k-i|^2} e^{\frac{t}{2} |\re 8ik^3|} 
	\\ 
 &\leq \frac{C}{|k-i|^2} e^{\frac{t}{2} |\re \Phi(\zeta, k)|}, \qquad \zeta \in \mathcal{I}, \quad t > 0, \quad k \in \bar{D}_1,
\end{align*}
and
\begin{align*}\nonumber
|f_r(t, k)| & \leq \frac{1}{|k-i|^2} \int_{-\infty}^{-\frac{t}{2}} s^2 |\hat{F}(s)| s^{-2} ds
 \leq \frac{C}{|k-i|^2}  \| s^2 \hat{F}(s)\|_{L^2(\R)} \sqrt{\int_{-\infty}^{-\frac{t}{2}} s^{-4} ds}  
 	\\ 
&  \leq \frac{C}{1 + |k|^2}  t^{-3/2}, \qquad \zeta \in \mathcal{I}, \quad t > 0, \quad k \in \partial D_1.
\end{align*}
Hence the $L^1$, $L^2$, and $L^\infty$ norms of $f_r$ on $\partial D_1$ are $O(t^{-3/2})$. The symmetry $f(k) = \overline{f(-\bar{k})}$ implies $F(\phi) = \overline{F(-\phi)}$. Thus the Fourier transform $\hat{F}(s)$ is real valued, which leads to the symmetries in (\ref{hsymmetries}).
Letting
\begin{align*}
& h_{a}(t, k) = f_0(k) + f_a(t, k), \qquad t > 0, \quad k \in \bar{D}_1,
	\\
& h_{r}(t, k) = f_r(t, k), \qquad t > 0, \quad k \in \partial D_1,
\end{align*}
we find a decomposition of $h$ with the properties listed in the statement of the lemma. 

\proofendcontinue

\medskip
{\bf Step 2: Deform.}
Lemma \ref{decompositionlemma} implies
$$h_{a}(t, \cdot) e^{t\Phi(\zeta, \cdot)} \in \dot{E}^2(D_1) \cap E^\infty(D_1)$$
for each $\zeta \in \mathcal{I}$ and each $t > 0$, where $E^\infty(D_1)$ denotes the space of bounded analytic functions in $D_1$. 
Thus, $M$ satisfies the $L^2$-RH problem (\ref{RHM}) iff the function $M^h(x,t,k)$ defined by
\begin{align}\label{checkMdef}
M^h(x,t,k) = \begin{cases}  
M(x,t,k)  \begin{pmatrix} 1 & 0 \\ -h_{a}(t, k) e^{t\Phi(\zeta, k)} & 1 \end{pmatrix}, & k \in D_1, \\
M(x,t,k)  \begin{pmatrix} 1 & -\overline{h_{a}(t, \bar{k})} e^{-t\Phi(\zeta, k)} \\ 0 & 1 \end{pmatrix}, & k \in D_4, \\
M(\zeta, t,k), & \text{otherwise},
\end{cases}
\end{align}
satisfies the $L^2$-RH problem
\begin{align}\label{RHMcheck}
\begin{cases}
M^h(x, t, \cdot) \in I + \dot{E}^2(\C \setminus \Sigma),\\
M^h_+(x,t,k) = M^h_-(x, t, k) J^h(x, t, k) \quad \text{for a.e.} \ k \in \Sigma,
\end{cases}
\end{align}
where 
\begin{align}\label{hatJdef}
&J^h(x,t,k) = \begin{cases} 
 \begin{pmatrix} 1 & 0 \\ h_{r}(t, k) e^{t\Phi(\zeta, k)} & 1 \end{pmatrix}, & k \in \partial D_1,
	\\
 \begin{pmatrix} 1 & - \overline{r(\bar{k})} e^{-t\Phi(\zeta, k)} \\
r(k) e^{t\Phi(\zeta, k)} & 1 - |r(k)|^2 \end{pmatrix},& k \in \R,
	\\ 
 \begin{pmatrix} 1 & \overline{h_{r}(t, \bar{k})} e^{-t\Phi(\zeta, k)} \\ 0 & 1 \end{pmatrix}, 	& k \in \partial D_4.
\end{cases}
\end{align}
Part $(c)$ of Lemma \ref{decompositionlemma} implies that the $L^1, L^2$, and $L^\infty$ norms of $J^h(x, t, \cdot) - I$ are $O(t^{-3/2})$ on $\partial D_1 \cup \partial D_4$.

\medskip
{\bf Step 3: Conjugate.}
In order to apply the steepest descent result of Theorem \ref{steepestdescentth}, we need to transform the RH problem in such a way that the jump matrix has decay everywhere as $t \to \infty$ except near the two stationary points. This can be achieved by performing an appropriate triangular factorization of the jump matrix followed by a contour deformation \cite{DZ1993}. For $k \in (-k_0,k_0)$, it is easy to achieve an appropriate factorization. We next show that an appropriate factorization can be achieved also for $k \in (-\infty, -k_0) \cup (k_0, \infty)$ by conjugating the RH problem (\ref{RHMcheck}).

Let
$$\Delta(\zeta, k) = \begin{pmatrix} \delta(\zeta, k)^{-1} & 0 \\  0 & \delta(\zeta, k)   \end{pmatrix},$$
where
\begin{align}\label{deltadef}
 \delta(\zeta, k) = e^{-\frac{1}{2\pi i} (\int_{-\infty}^{-k_0} + \int_{k_0}^\infty) \ln(1- |r(s)|^2) \frac{ds}{s - k}}, \qquad  k \in \C \setminus \R.
\end{align}
The function $\delta$ satisfies the following jump condition across the real axis:
\begin{align}\label{deltajump}
 \delta_+(\zeta, k) = \begin{cases} 
\frac{\delta_-(\zeta, k)}{1 - |r(k)|^2}, & |k| > k_0, \\
\delta_-(\zeta, k), & |k| < k_0, 
\end{cases} \qquad k \in \R.
\end{align}
Moreover, the symmetry $r(k) = \overline{r(-\bar{k})}$ implies that
\begin{align}\label{chideltasymm}
\delta(\zeta, k) = \overline{\delta(\zeta, -\bar{k})} =  \delta(\zeta, -k)^{-1}.
\end{align}
It follows that $\Delta$  obeys the symmetries
\begin{align}\label{Deltasymm}
  \Delta(\zeta, k) = \overline{\Delta(\zeta, -\bar{k})} =  \Delta(\zeta, -k)^{-1}.
\end{align}

\begin{lemma}\label{Deltalemma}
The $2 \times 2$-matrix valued function $\Delta(\zeta, k)$ satisfies
$$\Delta(\zeta, \cdot), \Delta(\zeta, \cdot)^{-1} \in I + \dot{E}^2(\C \setminus \R) \cap E^\infty(\C \setminus \R),$$
for each $\zeta \in \mathcal{I}$.
\end{lemma}
\proofbegin
First note that
\begin{align}\label{deltarep}
\delta(\zeta, k) = \bigg(\frac{k-k_0}{k + k_0}\bigg)^{i\nu} e^{\chi(\zeta, k)},
\end{align}
where $\nu(\zeta)$ is given by (\ref{nudef}),
\begin{align}\label{chidef}
\chi(\zeta, k) = - \frac{1}{2\pi i}\int_\R \psi(\zeta, s) \frac{ds}{s - k},
\end{align}
and the function $\psi(\zeta,s)$ is defined by (\ref{psidef}).
Since $\psi(\zeta, \cdot) \in H^1(\R)$, we have $\chi(\zeta, \cdot) \in E^\infty(\C \setminus \R)$ for each $\zeta \in \mathcal{I}$; in fact, see Lemma 23.3 in \cite{BDT1988},
\begin{align}\label{chibound}
\sup_{\zeta \in \mathcal{I}} \sup_{z \in \C \setminus \R} |\chi(\zeta, z)| \leq \sup_{\zeta \in \mathcal{I}} \|\psi(\zeta, \cdot)\|_{H^1(\R)} < \infty.
\end{align}
Hence $\delta(\zeta, \cdot), \delta(\zeta, \cdot)^{-1} \in E^\infty(\C \setminus \R)$ for each $\zeta \in \mathcal{I}$.

It remains to show that $\delta(\zeta, \cdot), \delta(\zeta, \cdot)^{-1} \in 1 + \dot{E}^2(\C \setminus \R)$.
Since $\psi(\zeta, \cdot) \in L^2(\R)$, $\chi(\zeta, \cdot) \in \dot{E}^2(\C \setminus \R)$ for each $\zeta \in \mathcal{I}$.
Let $C_n$ be the image of the circle  $|w| = 1 - \frac{1}{n}$ under the linear fractional transformation $z \mapsto - i\frac{z + i}{z - i}$, which maps the unit disk onto the upper half plane. 
Then $\sup_n \|\chi(\zeta, \cdot)\|_{L^2(C_n)} < \infty$. Using the inequality \begin{align}\label{ewminus1estimate}  
  |e^w - 1| \leq |w| \max(1, e^{\re w}), \qquad w \in \C,
\end{align}
we obtain
\begin{align*}
\sup_n \|\delta(\zeta, \cdot) - 1\|_{L^2(C_n)} 
& \leq \sup_n\bigg\|\bigg(\frac{k-k_0}{k + k_0}\bigg)^{i\nu} e^{\chi} - e^{\chi}\bigg\|_{L^2(C_n)} 
+  \sup_n \|e^{\chi} - 1\|_{L^2(C_n)} 
	\\
& \leq C\sup_n \|e^{\chi}\|_{L^\infty(C_n)} 
+ \sup_n \| \chi e^{|\re \chi|} \|_{L^2(C_n)}  < \infty.
\end{align*}
It follows that 
$\delta(\zeta, \cdot) \in 1 + \dot{E}^2(\C \setminus \R)$. A similar argument gives $\delta(\zeta, \cdot)^{-1} \in 1 + \dot{E}^2(\C \setminus \R)$.
\proofendcontinue

By Lemma \ref{Deltalemma}, $M^h(x,t,k)$ satisfies the $L^2$-RH problem (\ref{RHMcheck}) iff the function $\tilde{M}$ defined by
$$\tilde{M}(x,t,k) = M^h(x,t,k)\Delta(\zeta, k)$$
satisfies the $L^2$-RH problem
\begin{align}\label{RHMtilde}
\begin{cases}
\tilde{M}(x, t, \cdot)  \in I + \dot{E}^2(\C \setminus \Sigma),\\
\tilde{M}_+(x,t,k) = \tilde{M}_-(x, t, k) \tilde{J}(x, t, k)  \quad \text{for a.e.} \ k \in \Sigma,
\end{cases}
\end{align}
where
\begin{align} \nonumber
\tilde{J}&(x,t,k) = \Delta_-^{-1}(\zeta, k) J^h(x,t,k) \Delta_+(\zeta, k)
	\\ \label{tildeJdef}
& =\begin{cases}
 \begin{pmatrix} 1 & 0 \\ \delta(\zeta, k)^{-2} h_{r}(t, k) e^{t\Phi(\zeta,k)} & 1 \end{pmatrix}, \qquad k \in \partial D_1,
 	\\
 \begin{pmatrix} \frac{\delta_-(\zeta, k)}{\delta_+(\zeta, k)}	& -\delta_-(\zeta, k)\delta_+(\zeta, k)  \overline{r(\bar{k})}e^{-t\Phi(\zeta,k)}	\\
  \frac{1}{\delta_-(\zeta, k)\delta_+(\zeta, k)} r(k) e^{t\Phi(\zeta,k)}	& \frac{\delta_+(\zeta, k)}{\delta_-(\zeta, k)}(1 - |r(k)|^2) \end{pmatrix},  \qquad k \in \R,
	\\
\begin{pmatrix} 1 & \delta(\zeta, k)^2 \overline{h_{r}(t,\bar{k})} e^{-t\Phi(\zeta,k)} \\ 0 & 1 \end{pmatrix}, \qquad k \in \partial D_4.
  \end{cases}
\end{align}
In view of the jump (\ref{deltajump}) of $\delta(\zeta, k)$, this gives, for $k \in \R$,
\begin{align*}
\tilde{J}(x,t,k) = 
\begin{cases}
 \begin{pmatrix} 1 - |r(k)|^2	& - \delta_-(\zeta, k)^2 \frac{\overline{r(\bar{k})}}{1 - |r(k)|^2} e^{-t\Phi(\zeta,k)}	\\
 \delta_+(\zeta, k)^{-2} \frac{r(k)}{1 - |r(k)|^2} e^{t\Phi(\zeta,k)}	& 1 \end{pmatrix}, 	& |k| > k_0,	\\
 \begin{pmatrix} 1 & -\delta(\zeta, k)^2\overline{r(\bar{k})} e^{-t\Phi(\zeta,k)}	\\
    \delta(\zeta, k)^{-2} r(k) e^{t\Phi(\zeta,k)}	& 1 - |r(k)|^2 \end{pmatrix},   & |k| < k_0. 
\end{cases}
\end{align*}

The upshot of the above conjugation is that we can now factorize the jump matrix as follows:
\begin{align}\label{tildeJonR}
  \tilde{J} = 
\begin{cases}
 B_u^{-1} B_l,   & |k| > k_0, \quad k \in \R, \\
 b_l^{-1}b_u 	& |k| < k_0,  \quad k \in \R, 
\end{cases}
\end{align}
where
\begin{align}\nonumber
& B_l =  \begin{pmatrix} 1 & 0	\\
  \delta_+(\zeta, k)^{-2} r_1(k) e^{t\Phi(\zeta,k)}	& 1 \end{pmatrix}, 
	 \qquad
b_u = \begin{pmatrix} 1 & -\delta(\zeta, k)^2 r_2(k) e^{-t\Phi(\zeta,k)}	\\
0 & 1  \end{pmatrix}, 
	\\ \label{Bbdef}
&b_l = \begin{pmatrix} 1 &0 \\
 -  \delta(\zeta, k)^{-2} r_3(k) e^{t\Phi(\zeta,k)} & 1 \end{pmatrix}, 
 	\qquad
B_u =  \begin{pmatrix} 1	& \delta_-(\zeta, k)^2 r_4(k) e^{-t\Phi(\zeta,k)}	\\
0	& 1 \end{pmatrix},
\end{align}
and the functions $\{r_j(k)\}_1^4$ are defined by
\begin{align*}
& r_1(k) = \frac{r(k)}{1 - r(k)\overline{r(\bar{k})}}, \qquad
r_2(k) = \overline{r(\bar{k})},
	\\
& r_3(k) = r(k), \qquad
r_4(k) = \frac{\overline{r(\bar{k})}}{1 - r(k)\overline{r(\bar{k})}}.
\end{align*}
Our next goal is to deform the contour near $\R$. However, we first need to introduce analytic approximations of $\{r_j(k)\}_1^4$.

\medskip
{\bf Step 4: Introduce analytic approximations of $\{r_j(k)\}_1^4$.}
We introduce open sets $U_j = U_j(\zeta)$, $j = 1, \dots, 4$, as in Figure \ref{Omegas2.pdf} so that 
\begin{align}\label{Ujdef}
\{k \, | \, \re \Phi(\zeta, k) < 0\} = U_1 \cup U_3, \qquad \{k \, | \, \re \Phi(\zeta, k) > 0\} = U_2 \cup U_4.
\end{align}
\begin{figure}
\begin{center}
\bigskip
 \begin{overpic}[width=.55\textwidth]{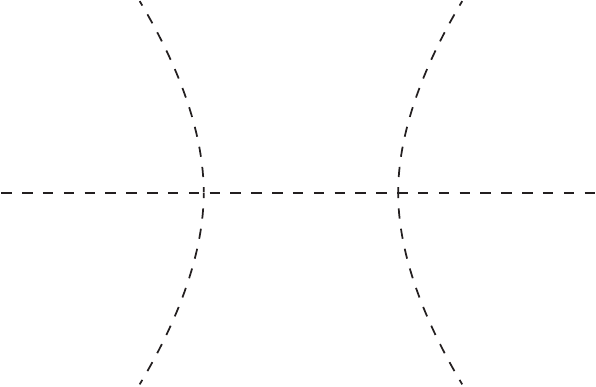}
 \put(85,46){$U_1$}
 \put(48,50){$U_2$}
 \put(48,12){$U_3$}
 \put(85,16){$U_4$}
 \put(10,46){$U_1$}
  \put(10,16){$U_4$}
 \put(24.5,27.5){$-k_0$}
 \put(69,27.5){$k_0$}
 \end{overpic}
   \begin{figuretext}\label{Omegas2.pdf}
      The domains $\{U_j\}_1^4$ in the complex $k$-plane. The dashed curves are the curves on which $\re \Phi = 0$.
      \end{figuretext}
   \end{center}
\end{figure}

\begin{figure}
\begin{center}
 \begin{overpic}[width=.65\textwidth]{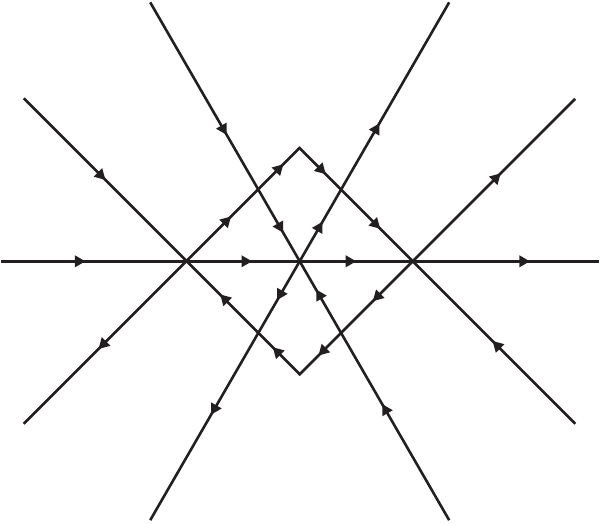}
      \put(67,38){\small $k_0$}
      \put(27.5,38){\small $-k_0$}
      \put(85,49){\small $V_1$}
      \put(68,58){\small $V_2$}
      \put(56,46){\small $V_3$}
      \put(56,38){\small $V_4$}
      \put(68,26){\small $V_5$}
      \put(85,35){\small $V_6$}
      \put(48,65){\small $V_2$}
      \put(48,52){\small $V_3$}
      \put(48,33){\small $V_4$}
      \put(48,20){\small $V_5$}
      \put(11,49){\small $V_1$}
      \put(28,58){\small $V_2$}
      \put(39,46){\small $V_3$}
      \put(39,38){\small $V_4$}
      \put(28,26){\small $V_5$}
      \put(11,35){\small $V_6$}
      \put(101,42.2){\small $\Gamma$}      
       \end{overpic}
    \qquad \qquad
     \begin{figuretext}\label{Gamma.pdf}
       The  contour  $\Gamma$ and the open sets $\{V_j\}_1^{6}$.
     \end{figuretext}
     \end{center}
\end{figure}

The following lemma describes how to decompose $r_j$, $j = 1, \dots, 4$, into an analytic part $r_{j,a}$ and a small remainder $r_{j,r}$. A proof can be found in \cite{Lnonlinearsteepest}.

\begin{lemma}\label{decompositionlemma2}
There exist decompositions
\begin{align*}
r_j(k) =
\begin{cases}
 r_{j,a}(x, t, k) + r_{j,r}(x, t, k), \qquad j = 1,4, \quad |k| > k_0,  \quad k \in \R,
	\\
 r_{j,a}(x, t, k) + r_{j,r}(x, t, k), \qquad j = 2,3, \quad |k| < k_0,  \quad k \in \R,
\end{cases}
\end{align*}
where the functions $\{r_{j,a}, r_{j,r}\}_{j=1}^4$ have the following properties:
\begin{enumerate}[$(a)$]
\item For each $\zeta \in \mathcal{I}$ and each $t > 0$, $r_{j,a}(x, t, k)$ is defined and continuous for $k \in \bar{U}_j$ and analytic for $k \in U_j$, $j = 1, \dots, 4$.

\item The functions $\{r_{j,a}\}_1^4$ satisfy, for each $K > 0$,
\begin{align}\nonumber
& |r_{j, a}(x, t, k) - r_j(k_0)| \leq 
C_K |k - k_0| e^{\frac{t}{4}|\re \Phi(\zeta,k)|}, 
	\\ \label{rjaestimatea}
& \hspace{4cm} k \in \bar{U}_j, \quad |k| \leq K, \quad \zeta \in \mathcal{I}, \quad t > 0,  \quad j = 1, \dots, 4, 
\end{align}
where the constant $C_K$ is independent of $\zeta, t, k$ but may depend on $K$.

\item The functions $r_{1,a}$ and $r_{4,a}$ satisfy
\begin{align}\label{rjaestimateb}
& |r_{j, a}(x, t, k)| \leq \frac{C}{1 + |k|} e^{\frac{t}{4}|\re \Phi(\zeta,k)|}, \qquad  k \in \bar{U}_j, \quad \zeta \in \mathcal{I}, \quad t > 0, \quad j = 1, 4, 
\end{align}
where the constant $C$ is independent of $\zeta, t, k$.

\item The $L^1, L^2$, and $L^\infty$ norms on $(-\infty, -k_0) \cup (k_0, \infty)$ of the functions $r_{1,r}(x, t, \cdot)$ and $r_{4,r}(x, t, \cdot)$ are $O(t^{-3/2})$ as $t \to \infty$ uniformly with respect to $\zeta \in \mathcal{I}$.

\item The $L^1, L^2$, and $L^\infty$ norms on $(-k_0, k_0)$ of the functions $r_{2,r}(x, t, \cdot)$ and $r_{3,r}(x, t, \cdot)$ are $O(t^{-3/2})$ as $t \to \infty$ uniformly with respect to $\zeta \in \mathcal{I}$.

\item The following symmetries are valid:
\begin{align}\label{rsymmetries}
r_{j,a}(\zeta, t, k) = \overline{r_{j,a}(\zeta, t, -\bar{k})}, \quad
r_{j,r}(\zeta, t, k) = \overline{r_{j,r}(\zeta, t, -\bar{k})}, \qquad j = 1, \dots, 4.
\end{align}
\end{enumerate}
\end{lemma}

\medskip
{\bf Step 5: Deform again.}
Let $\Gamma$ be the contour consisting of $\Sigma$ together with the four half-lines 
$$\big\{k_0 + u e^{\pm \frac{i\pi}{4}} \, |  -\sqrt{2} k_0 < u < \infty\big\}, \qquad
\big\{-k_0 + u e^{\pm \frac{3i\pi}{4}} \, | -\sqrt{2} k_0 < u < \infty\big\},$$
oriented as in Figure \ref{Gamma.pdf}. Let $\{V_j\}_1^{6}$ be the open sets shown in Figure \ref{Gamma.pdf}. Write
$$B_l = B_{l,r}B_{l,a}, \qquad
b_u = b_{u,r}b_{u,a}, \qquad
b_l = b_{l,r}b_{l,a}, \qquad
B_u = B_{u,r}B_{u,a},$$
where $\{B_{l,a}, b_{u,a}, b_{l,a}, B_{u,a}\}$ and $\{B_{l,r}, b_{u,r}, b_{l,r}, B_{u,r}\}$ denote the matrices $\{B_{l}, b_{u}, b_{l}, B_{u}\}$ with $\{r_j(k)\}_1^4$ replaced with $\{r_{j,a}(k)\}_1^4$ and $\{r_{j,r}(k)\}_1^4$, respectively.
The estimates (\ref{rjaestimatea}) and (\ref{rjaestimateb}) imply that 
\begin{align*}
\begin{cases}
B_{l,a}(x,t,\cdot), B_{l,a}^{-1}(x,t,\cdot) \in I + (\dot{E}^2 \cap E^\infty)(V_1),
	\\
b_{u,a}(x,t,\cdot), b_{u,a}^{-1}(x,t,\cdot) \in I + (\dot{E}^2 \cap E^\infty)(V_3),
	\\
b_{l,a}(x,t,\cdot), b_{l,a}^{-1}(x,t,\cdot) \in I + (\dot{E}^2 \cap E^\infty)(V_4),
	\\
B_{u,a}(x,t,\cdot), B_{u,a}^{-1}(x,t,\cdot) \in I + (\dot{E}^2 \cap E^\infty)(V_6),
\end{cases}
\end{align*}
for each $\zeta \in \mathcal{I}$ and each $t > 0$.
Hence $\tilde{M}$ satisfies the $L^2$-RH problem (\ref{RHMtilde}) iff $m(x,t,k)$ defined by
\begin{align*}
m(x,t,k) = \begin{cases}  
\tilde{M}(x,t,k) B_{l,a}(x,t,k)^{-1}, & k \in V_1, \\
\tilde{M}(x,t,k) b_{u,a}(x,t,k)^{-1}, & k \in V_3, \\
\tilde{M}(x,t,k) b_{l,a}(x,t,k)^{-1}, & k \in V_4, \\
\tilde{M}(x,t,k) B_{u,a}(x,t,k)^{-1}, & k \in V_6, \\
\tilde{M}(x,t,k), & \text{elsewhere},
\end{cases}
\end{align*}
satisfies the $L^2$-RH problem
\begin{align}\label{RHm2}
\begin{cases} m(x, t, \cdot) \in I + \dot{E}^2(\C \setminus \Gamma), \\
m_+(x, t, k) = m_-(x, t, k) v(x, t, k) \quad \text{for a.e.} \ k \in \Gamma, 
\end{cases} 
\end{align}
where, in view of (\ref{tildeJdef}) and (\ref{tildeJonR}), the jump matrix $v$ is given by
\begin{align}\label{vdef}
v = \begin{cases}
B_{l,a} = \begin{pmatrix} 1 & 0	\\
  \delta(\zeta, k)^{-2} r_{1,a}(x,t,k) e^{t\Phi(\zeta,k)}	& 1 \end{pmatrix}, & k \in \bar{V}_1 \cap \bar{V}_2, 
  	\\
b_{u,a} = \begin{pmatrix} 1 & -\delta(\zeta, k)^2 r_{2,a}(x,t,k) e^{-t\Phi(\zeta,k)}	\\
0 & 1  \end{pmatrix}, & k \in \bar{V}_2 \cap \bar{V}_3, 
	\\
b_{l,a} = \begin{pmatrix} 1 &0 \\
 -  \delta(\zeta, k)^{-2} r_{3,a}(x,t,k) e^{t\Phi(\zeta,k)} & 1 \end{pmatrix}, & k \in \bar{V}_4 \cap \bar{V}_5,
 	\\
B_{u,a} = \begin{pmatrix} 1	& \delta(\zeta, k)^2 r_{4,a}(x,t,k) e^{-t\Phi(\zeta,k)}	\\
0	& 1 \end{pmatrix}, & k \in \bar{V}_5 \cap \bar{V}_6.
	\\
B_{u,r}^{-1} B_{l,r}, & k \in \bar{V}_1 \cap \bar{V}_6, 
	\\
b_{l,r}^{-1}b_{u,r}, & k \in \bar{V}_3 \cap \bar{V}_4, 
 	\\
\begin{pmatrix} 1 & 0 \\ \delta(\zeta,k)^{-2}  h_{r}(t, k) e^{t\Phi(\zeta, k)} & 1 \end{pmatrix}, & k \in \partial D_1,
 	\\
\begin{pmatrix} 1 & \delta(\zeta,k)^2 \overline{h_r(t, \bar{k})} e^{-t\Phi(\zeta, k)} \\0 & 1 \end{pmatrix}, & k \in \partial D_4.
\end{cases}
\end{align}
 
From the symmetries (\ref{hsymmetries}), (\ref{chideltasymm}), and (\ref{rsymmetries}), we infer that $v$ satisfies
\begin{align}\label{vsymm2}
v(x,t,k) = \overline{v(x, t, -\bar{k})}, \qquad k \in \Gamma.
\end{align}

\medskip
{\bf Step 6: Apply Theorem \ref{steepestdescentth}.}
We verify that Theorem \ref{steepestdescentth} can be applied to the interval $\mathcal{I} = (0, N]$, the contour $\Gamma$, and the jump matrix $v$ with 
\begin{align*} \nonumber
& \epsilon = \frac{k_0}{2}, \qquad
\rho = \epsilon \sqrt{-i \frac{\partial^2 \Phi}{\partial k^2}(\zeta, k_0)}
 = \epsilon \sqrt{48 k_0}, \qquad \tau = t\rho^2 = 12 k_0^3 t,
 	\\
&  q(\zeta) = e^{\frac{1}{\pi i}\dashint_\R \frac{\psi(\zeta, s) ds}{s-k_0}} r(k_0) e^{2i\nu(\zeta) \ln(2\sqrt{48}k_0^{3/2})}, \qquad \nu(\zeta) = -\frac{1}{2\pi} \ln(1 -  |r(k_0)|^2),
	\\ 
& \phi(\zeta, z) = \Phi\biggl(\zeta, k_0 - \frac{\epsilon}{\rho}z\biggr)
 = -16i k_0^3 + \frac{iz^2}{2} - \frac{i z^3}{12 \rho}.
\end{align*}

\begin{figure}
\begin{center}
\medskip
 \begin{overpic}[width=.45\textwidth]{Gamma.pdf}
     \put(48,66.5){$\Gamma$}
           \put(67,38){\tiny $k_0$}
      \put(27.2,38){\tiny $-k_0$}
   \end{overpic}
       \qquad 
 \begin{overpic}[width=.445\textwidth]{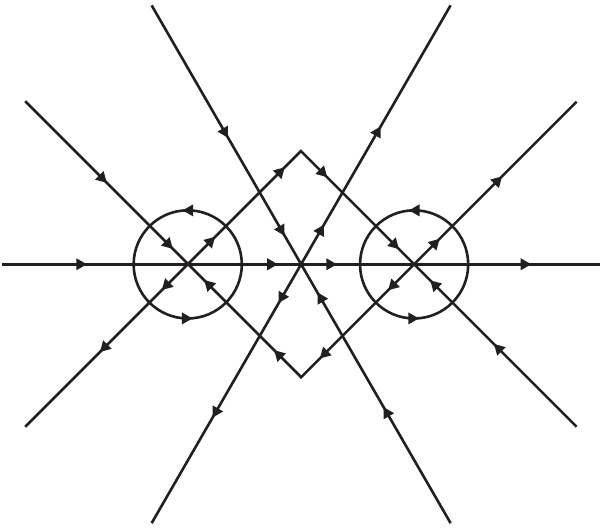}
      \put(48,67){$\hat{\Gamma}$}
            \put(67,38){\tiny $k_0$}
      \put(27.2,38){\tiny $-k_0$}
      \end{overpic}
     \begin{figuretext}\label{Gammahat.pdf}
       The contours $\Gamma$ and $\hat{\Gamma}$ satisfy the properties {\upshape ($\Gamma$1)-($\Gamma$4)} of Theorem \ref{steepestdescentth}.
     \end{figuretext}
     \end{center}
\end{figure}

The contours $\Gamma$ and $\hat{\Gamma}$ are shown in Figure \ref{Gammahat.pdf}.
The conditions ($\Gamma$1)-($\Gamma$4) are clearly satisfied. 
Since the contour $k_0^{-1}\hat{\Gamma}$ is independent of $\zeta$, a scaling argument shows that $\|\mathcal{S}_{\hat{\Gamma}}\|_{\mathcal{B}(L^2(\hat{\Gamma}))}$ is independent of $\zeta$. In particular, $\mathcal{S}_{\hat{\Gamma}}$ is uniformly bounded on $L^2(\hat{\Gamma})$.
Equation (\ref{winL1L2Linf}) follows from (\ref{vdef}) and the estimates in Lemmas \ref{decompositionlemma} and \ref{decompositionlemma2}.
Clearly $\det v = 1$. The symmetry condition (\ref{vsymm}) follows from (\ref{vsymm2}).

We next verify (\ref{wL12infty}).
Let $w = v - I$ and let $\Gamma'$ denote the contour obtained from $\Gamma$  by removing the crosses $\pm k_0 + X^\epsilon$. By parts $(d)$ and $(e)$ of Lemma \ref{decompositionlemma2}, the $L^1$ and $L^\infty$ norms of $w$ are $O(t^{-\frac{3}{2}}) \leq O(\epsilon \tau^{-1})$ on $\R$ uniformly with respect to $\zeta \in \mathcal{I}$.
By part $(c)$ of Lemma \ref{decompositionlemma}, the $L^1$ and $L^\infty$ norms of $w$ are $O(t^{-\frac{3}{2}}) \leq O(\epsilon \tau^{-1})$ on $\partial D_1 \cup \partial D_4$ uniformly with respect to $\zeta \in \mathcal{I}$.
Let $\gamma$ denote the part of $\Gamma'$ that belongs to the line $k_0 + \R e^{\frac{i\pi}{4}}$, i.e.
$$\gamma = \bigg\{k_0 + ue^{\frac{i\pi}{4}} \, \bigg| \, u \in \bigg(-\sqrt{2}k_0, -\frac{k_0}{2}\bigg] \cup \bigg[\frac{k_0}{2}, \infty\bigg) \bigg\}.$$ 
Let $k = k_0 + ue^{\frac{i\pi}{4}}$. Then
\begin{align}\label{rePhiest}
\re \Phi(\zeta, k) = -4u^2(6k_0 + \sqrt{2} u) < - 16 u^2 k_0 \qquad \text{for} \quad -\sqrt{2}k_0 < u < \infty.
\end{align}
Together with (\ref{rjaestimateb}) this yields
\begin{align*}
|r_{1,a}(x, t, k) e^{t \Phi(\zeta, k)}| & \leq Ce^{-\frac{3t}{4}|\re \Phi(\zeta,k)|}
 \leq Ce^{-12 tu^2k_0} 
 \leq \frac{C t u^2k_0 e^{-12 tu^2k_0}}{t u^2 k_0} 
 	\\
& 
\leq \frac{C}{t u^2 k_0}
\leq \frac{Ck_0^2}{u^2 \tau}, \qquad \frac{k_0}{2}< u < \infty.
\end{align*}
Similarly, by (\ref{rjaestimatea}),
\begin{align*}
|r_{3,a}(x, t, k) e^{t \Phi(\zeta, k)}| & \leq C e^{-\frac{3t}{4}|\re \Phi(\zeta,k)|}
\leq \frac{Ck_0^2}{u^2 \tau}, \qquad -\sqrt{2}k_0 < u < -\frac{k_0}{2}.
\end{align*}
Hence
\begin{align*}
& \|w\|_{L^1(\gamma)} \leq  \frac{Ck_0^2}{\tau} \bigg(\int_{-\sqrt{2}k_0}^{-\frac{k_0}{2}} + \int_{\frac{k_0}{2}}^\infty\bigg)  u^{-2} du= O(k_0 \tau^{-1}),
	\\
& \|w\|_{L^\infty(\gamma)} = O(\tau^{-1}).
\end{align*}
This shows that the estimates in (\ref{wL12infty}) hold also on $\gamma$. Since similar arguments apply to the remaining parts of $\Gamma'$, this  verifies (\ref{wL12infty}).

Equation (\ref{vdef}) implies that (\ref{vjdef}) and (\ref{smallcrossjump}) are satisfied with
\begin{align*}
\begin{cases}
R_1(\zeta, t, z) =  \delta(\zeta,k)^{-2} r_{3,a}(x,t,k)z^{2i\nu(\zeta)},
	\\
R_2(\zeta, t, z) = \delta(\zeta, k)^2 r_{4,a}(x,t,k) z^{-2i\nu(\zeta)},
	\\
R_3(\zeta, t, z) = \delta(\zeta,k)^{-2} r_{1,a}(x,t,k) z^{2i\nu(\zeta)},
	\\
R_4(\zeta, t, z) = \delta(\zeta,k)^2 r_{2,a}(x,t,k) z^{-2i\nu(\zeta)},
\end{cases}
\end{align*}
where $k$ and $z$ are related by $k = k_0 - \frac{\epsilon z}{\rho}$. 

It is clear that $\phi$ satisfies (\ref{phiassumptions}) and (\ref{Phiz3estimate}). 
The following estimate proves  (\ref{rephiestimateb}):
$$\re \phi(\zeta, z) = \frac{|z|^2}{2}\bigg(1 \pm \frac{|z|}{6\sqrt{2} \rho}\bigg) \geq \frac{|z|^2}{4}, \qquad z \in X_2^{\rho} \cup X_4^{\rho}, \quad \zeta \in \mathcal{I},$$
where the plus and minus signs are valid for $z \in X_2^{\rho}$ and $z \in X_4^{\rho}$ respectively. The proof of (\ref{rephiestimatea}) is similar.
Since $|q(\zeta)| = |r(k_0)|$, we have $\sup_{\zeta \in \mathcal{I}} |q(\zeta)| < 1$ and $\nu(\zeta) = -\frac{1}{2\pi} \ln(1 - |q(\zeta)|^2)$.

Finally, we show that given any $\alpha \in [1/2, 1)$, there exists an $L >0$ such that the inequalities (\ref{Lipschitzconditions}) hold. Let  $k = k_0 - \frac{\epsilon z}{\rho}$.
Using the expression (\ref{deltarep}) for $\delta(\zeta, k)$, we may write
$$R_1(\zeta, t, z) = e^{-2\chi(\zeta, k)} r_{3,a}(x,t,k) e^{2i\nu(\zeta) \ln((-k_0 - k)\sqrt{48k_0})} , \qquad z \in X_1.$$ 
Thus,
$$R_1(\zeta, t, 0) = e^{-2\chi_-(\zeta, k_0)} r_{3,a}(x,t,k_0)  e^{2i\nu(\zeta) (\ln(2\sqrt{48}k_0^{3/2}) + i\pi)}.$$
Now $r_{3,a}(x,t,k_0) = r(k_0)$ by (\ref{rjaestimatea}), and
$$\chi_-(\zeta, k_0) = - \frac{1}{2\pi i} \dashint_\R \frac{\psi(\zeta, s) ds}{s-k_0} - \pi \nu(\zeta)$$
by Sokhotski-Plemelj. Hence $R_1(\zeta, t, 0) = q(\zeta)$.
Similarly, we find 
$$R_2(\zeta, t, 0) = \frac{\overline{q(\zeta)}}{1 - |q(\zeta)|^2}, \qquad
R_3(\zeta, t, 0) = \frac{q(\zeta)}{1 - |q(\zeta)|^2}, \qquad
R_4(\zeta, t, 0) = \overline{q(\zeta)}.$$
We establish the estimate (\ref{Lipschitzconditions}) in the case of $z \in X_1^{\rho}$; the other cases are similar.
Note that $z \in X_1^{\rho}$ corresponds to $k \in k_0 + X_3^\epsilon$.

We have, for $z \in X_1^\rho$,
\begin{align*}
|R_1(\zeta, t, z) &- q(\zeta)| \leq 
\left|e^{-2\chi(\zeta, k)} - e^{-2\chi_-(\zeta, k_0)}\right| \left|r_{3,a}(x,t,k) e^{2i\nu(\zeta) \ln((-k_0 - k)\sqrt{48k_0})}\right|
	\\
& + \left| e^{-2\chi_-(\zeta,k_0)}\right| \left|r_{3,a}(x,t,k) - r_3(k_0)\right| \left| e^{2i\nu(\zeta) \ln((-k_0 - k)\sqrt{48k_0})}\right|
	\\
& + \left| e^{-2\chi_-(\zeta, k_0)} r_3(k_0)\right| \left|e^{2i\nu(\zeta) (\ln((k_0 + k)\sqrt{48k_0}) + i\pi)} - e^{2i\nu(\zeta) (\ln(2\sqrt{48}k_0^{3/2}) + i\pi)}\right|.
\end{align*}
The functions $e^{2i\nu(\zeta)\ln((-k_0 - k)\sqrt{48k_0})}$ and $e^{-2\chi_-(\zeta, k_0)}$ are uniformly bounded with respect to $k \in k_0 + X_3^\epsilon$ and $\zeta \in \mathcal{I}$. 
Moreover, employing the estimate 
$$|\re \phi(\zeta, ve^{\frac{i\pi}{4}} )| = \bigg|-\frac{v^2}{2}\bigg(1 - \frac{v}{\sqrt{72} \rho}\bigg)\bigg|
\leq \frac{2v^2}{3}, \qquad -\rho \leq v \leq \rho,$$
we see that equation (\ref{rjaestimatea}) yields
\begin{align*}
|r_{3, a}(x, t, k) - r_3(k_0)| 
& \leq C |k -k_0| e^{\frac{t}{4}|\re \Phi(\zeta,k)|} 
 = C \frac{\epsilon|z|}{\rho} e^{\frac{t}{4}|\re \phi(\zeta,z)|} 
 	\\
&  \leq C \frac{\epsilon|z|}{\rho} e^{\frac{t |z|^2}{6}}, \qquad z \in X_1^{\rho}.
\end{align*}
Thus,
\begin{align}\nonumber
 |R_1(\zeta, t, z) - q(\zeta)| \leq &\; C e^{\frac{t |z|^2}{6}}
 \left|e^{-2\chi(\zeta, k)} - e^{-2\chi_-(\zeta, k_0)}\right| + C \frac{\epsilon |z|}{\rho}e^{\frac{t|z|^2}{6}}
	\\ \label{R1minusq}
& + C \big|1 - e^{-2i\nu(\zeta)\ln(\frac{k_0 + k}{2k_0})} \big|, \qquad \zeta \in \mathcal{I}, \quad t > 0, \quad z \in X_1^{\rho}.
\end{align}
The estimate in (\ref{Lipschitzconditions}) for $z \in X_1^{\rho}$ now follows from the following lemma.

\begin{lemma}
The following inequalities are valid for all $\zeta \in \mathcal{I}$ and all $k \in k_0 + X_3^\epsilon$:
\begin{align}\label{term1}
& \left|e^{-2\chi(\zeta, k)} - e^{-2\chi_-(\zeta, k_0)}\right| \leq C |k-k_0| (1 + |\ln|k-k_0||), 
	\\ \label{term3}
& \left|1 - e^{-2i\nu(\zeta)\ln(\frac{k_0 + k}{2k_0})}\right|  \leq C k_0^{-1} |k-k_0|,
\end{align}
where the constant $C$ is independent of $\zeta$ and $k$. 
\end{lemma}
\proofbegin
We first prove that
\begin{align}\label{chiminuschi}
|\chi(\zeta, k) - \chi_-(\zeta, k_0)| \leq C |k-k_0|\big(1 + |\ln|k-k_0||\big), \qquad \zeta \in \mathcal{I}, \quad k \in k_0 + X_3^\epsilon.
\end{align}
Integration by parts in the definition (\ref{chidef}) of $\chi$ gives 
\begin{align}\label{chizetal}
 \chi(\zeta, k) = \frac{1}{2\pi i} \bigg(\int_{-\infty}^{-k_0} + \int_{k_0}^\infty\bigg) \ln(s-k) d\ln\bigl(1- |r(s)|^2\bigr).
\end{align}
Hence
$$|\chi(\zeta, k) - \chi_-(\zeta, k_0)| 
= \frac{1}{2\pi} \biggl| \bigg(\int_{-\infty}^{-k_0} + \int_{k_0}^\infty\bigg) \ln\biggl(\frac{s - k}{s - k_0}\biggr) d\ln\bigl(1- |r(s)|^2\bigr) \biggr|.$$
This gives
\begin{align}\nonumber
& |\chi(\zeta, k) - \chi_-(\zeta, k_0)| 
 \leq C \bigg(\int_{-\infty}^{-k_0} + \int_{k_0}^\infty\bigg) \biggl| \ln\frac{s - k}{s - k_0}\biggr| \frac{ds}{1 + |s|^2}
	\\ \label{chichichi}
&\leq  C \int_{2k_0}^\infty \biggl| \ln\biggl(1 + \frac{k - k_0}{u}\biggr) \biggr| \frac{du}{1 + (u - k_0)^2}
+ C \int_0^\infty \biggl| \ln\biggl(1 - \frac{k - k_0}{u}\biggr) \biggr|\frac{du}{1 + u^2}.
\end{align}
The change of variables $v =|k - k_0|/u$ yields
\begin{align*}
\int_0^\infty \biggl| \ln\biggl(1 - \frac{k - k_0}{u}\biggr) \biggr| \frac{du}{1 + u^2}
& = |k-k_0| \int_0^\infty \bigl| \ln(1 + v e^{\frac{i\pi}{4}}) \bigr| \frac{dv}{v^2 + |k - k_0|^2}.
\end{align*}
Since 
$$\bigl| \ln(1 + v e^{\frac{i\pi}{4}}) \bigr|  \leq C \begin{cases}
v, & 0 \leq v \leq 2, 
	\\
\ln v, & 2 \leq v < \infty,
\end{cases}$$
we conclude that
\begin{align*}
\int_0^\infty \biggl| \ln\biggl(1 - \frac{k - k_0}{u}\biggr) \biggr| \frac{du}{1 + u^2}
& \leq C|k - k_0|\bigg(\int_0^2 \frac{vdv}{v^2 + |k - k_0|^2} + \int_2^\infty \frac{\ln v}{v^2}dv\bigg)
	\\
& \leq C |k-k_0| (|\ln|k-k_0|| + C), \qquad k \in k_0 + X_3^\epsilon.
\end{align*}
A similar argument applies to the first term on the right-hand side of (\ref{chichichi}). 
This proves (\ref{chiminuschi}). 

Using the inequality (\ref{ewminus1estimate}) together with (\ref{chibound}) and (\ref{chiminuschi}), we estimate
\begin{align*}
|e^{-2\chi(\zeta, k)} - e^{-2\chi_-(\zeta, k_0)}|
& \leq |e^{-2\chi_-(\zeta, k_0)}| |e^{-2[\chi(\zeta, k) - \chi_-(\zeta, k_0)]} - 1 |
	\\
& \leq C |\chi(\zeta, k) - \chi_-(\zeta, k_0)| e^{2|\re(\chi(\zeta, k) - \chi_-(\zeta, k_0))|}
	\\
& \leq C |k-k_0| (1+ |\ln|k-k_0||), \qquad \zeta \in \mathcal{I}, \quad k \in 1+ X_3^\epsilon.
\end{align*}
This proves (\ref{term1}).

By (\ref{ewminus1estimate}),
\begin{align*}
\left|1 - e^{-2i\nu(\zeta)\ln(\frac{k_0 + k}{2k_0})}\right| & \leq \bigg|2\nu(\zeta)\ln\bigg(\frac{k_0 + k}{2k_0}\bigg)\bigg| e^{|\re(2i\nu(\zeta)\ln(\frac{k_0 + k}{2k_0}))|}
	\\
& \leq C \bigg|\ln\biggl(1 + \frac{k-k_0}{2k_0}\biggr)\bigg| \leq C k_0^{-1} |k-k_0|, \quad \zeta \in \mathcal{I},
\quad k \in k_0 + X_3^\epsilon, 
\end{align*}
which proves (\ref{term3}).
\proofendcontinue

\medskip
{\bf Step 7: Find asymptotics.}
As a consequence of Theorem \ref{steepestdescentth}, the $L^2$-RH problem (\ref{RHm2}) for $m$ has a unique solution and the limit in (\ref{limlm12}) exists for all sufficiently large $\tau$. Since the RH problems for $M$ and $m$ are equivalent, the $L^2$-RH problem (\ref{RHM}) for $M$ also has a unique solution for all sufficiently large $\tau$. Moreover, we conclude that the limit (\ref{ulim}) defining $u(x,t)$ exists whenever $\tau$ is large enough and
\begin{align*}
u(x,t) & = -2i\ntlim_{k\to\infty} (kM(x,t,k))_{12} = -2i \ntlim_{k\to\infty} (km(x,t,k))_{12}.
\end{align*}
Equation (\ref{limlm12}) of Theorem \ref{steepestdescentth} then yields
\begin{align*}
u(x,t) & = -\frac{4 \epsilon \re \beta(\zeta, t)}{\sqrt{\tau}} + O\bigl(\epsilon \tau^{-\frac{1+\alpha}{2}}\bigr)
 = -\frac{\re \beta(\zeta, t)}{\sqrt{3t k_0}} + O\bigl(\epsilon \tau^{-\frac{1+\alpha}{2}}\bigr),	
\end{align*}
where $\beta(\zeta, t)$ is defined by (\ref{betadef}). We have
$$\re \beta(\zeta,t) = \sqrt{\nu(\zeta)}\cos\bigg(\frac{\pi}{4} - \arg q(\zeta) + \arg \Gamma(i\nu(\zeta))
 + 16 tk_0^3 - \nu(\zeta)\ln t\bigg).$$
where
$$\arg q(\zeta) = 2\nu(\zeta) \ln(2k_0 \sqrt{48k_0}) + \arg{r(k_0)} - \frac{1}{\pi} \dashint_\R \frac{\psi(\zeta, s)ds}{s - k_0}.$$
Thus,
\begin{align*}
u(x,t) & = -\sqrt{\frac{\nu(\zeta)}{3tk_0}} \cos(16tk_0^3 - \nu(\zeta) \ln(192tk_0^3) + \phi(\zeta)) + O\bigl(\epsilon \tau^{-\frac{1+\alpha}{2}}\bigr),
\end{align*}
where $\phi(\zeta)$ is defined by (\ref{phizetadef}). This proves (\ref{ufinal}) and completes the proof of Theorem \ref{mainth1}.
\proofend

\section{A model RH problem}\nequation\label{painlevesec}
The remainder of this paper is devoted to deriving the long-time asymptotics for the mKdV equation (\ref{mkdv}) in the self-similar sector (\ref{sectorIVdef}). 
The goal of the present section is to prove Theorem \ref{mZth} which expresses the large $z$ behavior of the solution of a model RH problem (see equation (\ref{RHmZ})) in terms of the solution of the Painlev\'e II equation. The result of Theorem \ref{mZth} will be important for the proof of Theorem \ref{steepestdescentthIV}.

Let $Z$ denote the contour $Z = Z_1 \cup Z_2 \cup Z_3$, where the line segments
\begin{align} \nonumber
&Z_1 = \bigl\{1+ re^{\frac{i\pi}{6}}\, \big| \, 0 \leq r < \infty\bigr\} \cup \bigl\{-1 + re^{\frac{5i\pi}{6}}\, \big| \, 0 \leq r < \infty\bigr\},  
	\\ \nonumber
&Z_2 = \bigl\{-1 + re^{-\frac{5i\pi}{6}}\, \big| \, 0 \leq r < \infty\bigr\} \cup \bigl\{1 + re^{-\frac{i\pi}{6}}\, \big| \, 0 \leq r < \infty\bigr\}, 
	\\ \label{Zdef}
& Z_3 = \bigl\{r\, \big| \, -1 \leq r \leq 1\bigr\}
\end{align}
are oriented as in Figure \ref{Z.pdf}. 
\begin{figure}
\begin{center}
\bigskip
 \begin{overpic}[width=.6\textwidth]{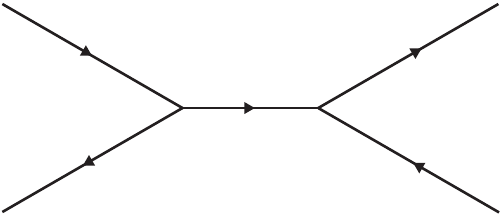}
 \put(78,36){$Z_1$}
 \put(17,36){$Z_1$}
 \put(17,4){$Z_2$}
 \put(78,4){$Z_2$}
 \put(47,25){$Z_3$}
 \put(62,16.5){$1$}
 \put(33,16.5){$-1$}
 \end{overpic}
   \bigskip
   \begin{figuretext}\label{Z.pdf}
      The contour $Z$. 
      \end{figuretext}
   \end{center}
\end{figure}
Given a positive number $z_0 > 0$, let $z_0 Z = \{z_0 z \in \C\, | \,  z \in Z\}$ denote the contour $Z$ scaled by the factor $z_0$. 
It turns out that the long time asymptotics in the self-similar sector is related to the solution $m^Z$ of the following family of RH problems parametrized by $s \in i\R$, $y \in \R$,  $z_0 > 0$:
\begin{align}\label{RHmZ}
\begin{cases}
m^Z(s, y, \cdot, z_0) \in I + \dot{E}^2(\C \setminus z_0Z),\\
m_+^Z(s, y, z, z_0) =  m_-^Z(s, y, z, z_0) v^Z(s, y, z, z_0) \quad \text{for a.e.} \ z \in z_0Z,
\end{cases}
\end{align}
where the jump matrix $v^Z(s, y, z, z_0)$ is defined by
\begin{align}
v^Z(s, y, z, z_0) = \begin{cases}
 \begin{pmatrix} 1	& 0  \\
s e^{2i(yz + \frac{4z^3}{3})}  & 1 \end{pmatrix}, &  z \in z_0Z_1, 
	\\
\begin{pmatrix} 1 & -se^{-2i(yz + \frac{4z^3}{3})}	\\
0	& 1 \end{pmatrix}, &   z \in z_0Z_2, 
  	\\
\begin{pmatrix} 1 & se^{-2i(yz + \frac{4z^3}{3})}\\
0 	& 1 \end{pmatrix}
\begin{pmatrix} 1 &0 \\
s e^{2i(yz + \frac{4z^3}{3})} 	& 1 \end{pmatrix}, &  z \in z_0Z_3.
\end{cases}
\end{align}
Note that $v^Z \to I$ exponentially fast as $z \to \infty$, $z \in z_0Z$.
Moreover, $v^Z$ obeys the symmetry
$$v^Z(s, y, z, z_0) = \sigma_3 \overline{v^Z(s, y, -\bar{z}, z_0)} \sigma_3.$$
The following theorem expresses the asymptotics of  $m^Z$ in terms of the solution $u^P$ of the Painlev\'e II equation 
\begin{align}\label{painleve2}
  u^P_{yy} = yu^P + 2(u^P)^3.
\end{align}

\begin{theorem}\label{mZth}
Let $s \in i \R$ be such that $|s|<1$. Let $y \in \R$ and $z_0 > 0$. Then the $L^2$-RH problem (\ref{RHmZ}) has a unique solution $m^Z(s, y, z, z_0)$ which satisfies
\begin{align}\label{mZasymptotics}
  m^Z(s, y, z, z_0) = I + \frac{1}{2z}\begin{pmatrix} 0 & u^P(y; s, 0, -s) \\ u^P(y; s, 0, -s) & 0 \end{pmatrix} + O\biggl(\frac{1}{z^2}\biggr), \qquad z \to \infty,
\end{align}  
where $u^P(\cdot; s, 0,-s)$ denotes the solution of the Painlev\'e II equation (\ref{painleve2}) corresponding to $(s_1, s_2, s_3) = (s,0,-s)$ according to (\ref{painlevebijection}), and the error term is uniform with respect to $z_0 > 0$, $\arg z \in [0, 2\pi]$, and $y$ in compact subsets of $\R$.
Moreover, if $K$ is a compact subset of $\R$, then
\begin{align}\label{mZbounded}
\sup_{z_0 > 0} \sup_{y \in K} \sup_{z \in \C\setminus z_0Z}  |m^Z(s,y,z,z_0)| < \infty.
\end{align}
\end{theorem}
\proofbegin
Uniqueness follows since $\det v = 1$.
Moreover, the RH problem (\ref{RHmZ}) for $m^Z$ can be transformed into the RH problem (\ref{RHmP}) for $m^P$ with jump data specified by $(s_1, s_2, s_3) = (s,0,-s)$.  
Indeed, let $\{\Omega_j\}_0^2$ be as in Figure \ref{OmegasXZ.pdf}. 
\begin{figure}
\begin{center}
\bigskip
 \begin{overpic}[width=.7\textwidth]{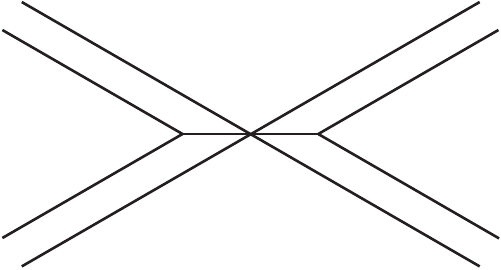}
 \put(34.4,24){$-z_0$}
 \put(62,24){$z_0$}
 \put(85,26){$\Omega_0$}
 \put(78,39.5){$\Omega_1$}
 \put(49,45){$\Omega_0$}
 \put(17,40){$\Omega_1$}
 \put(11,26){$\Omega_0$}
 \put(17,12){$\Omega_2$}
 \put(49,7){$\Omega_0$}
 \put(78.5,12.5){$\Omega_2$}
 \end{overpic}
   \bigskip
   \begin{figuretext}\label{OmegasXZ.pdf}
      The sets $\Omega_j$, $j = 0,1,2$. 
      \end{figuretext}
   \end{center}
\end{figure}
For each $y \in \R$ and $z_0 > 0$, we have
$$e^{2i(yz + \frac{4z^3}{3})} \in (\dot{E}^2 \cap E^\infty)(\Omega_1), \qquad
e^{-2i(yz + \frac{4z^3}{3})} \in (\dot{E}^2 \cap E^\infty)(\Omega_2).$$
Hence the function $m^P(y, z)$ defined by
$$m^P(y,z) = \begin{cases}
m^Z(s, y, z, z_0), & z \in \Omega_0,
	\\
m^Z(s, y, z, z_0)\begin{pmatrix} 1 & 0 \\ -s e^{2i(yz + \frac{4z^3}{3})} & 1 \end{pmatrix}, & z \in \Omega_1,
	\\
m^Z(s, y, z, z_0)\begin{pmatrix} 1 & s e^{-2i(yz + \frac{4z^3}{3})}  \\ 0 & 1 \end{pmatrix}, & z \in \Omega_2,
\end{cases}
$$
satisfies the RH problem (\ref{RHmP}) with jump data specified by $(s_1, s_2, s_3) = (s,0,-s)$ if and only if $m^Z$ satisfies the RH problem (\ref{RHmZ}). 
Since $s \in i\R$ and $|s| < 1$, the set $Y_{\mathcal{S}}$ in (\ref{mPasymptotics}) does not intersect $\R$ for $\mathcal{S} = (s,0,-s)$, see \cite{HM1980} and p. 51 of \cite{IN1986}. Thus (\ref{mZasymptotics}) follows from (\ref{mPasymptotics}). Equation (\ref{mZbounded}) follows from the properties of $m^P$. 
\proofend

\section{Another nonlinear steepest descent theorem}\nequation\label{steepsec2}
In the self-similar sector $\mathbb{S}$ defined in (\ref{sectorIVdef}), the asymptotics of the mKdV equation is characterized by the solution of a RH problem whose jump matrix has two critical points that merge at the origin as $t \to \infty$. In this section, we prove Theorem \ref{steepestdescentthIV} which provides an implementation of the nonlinear steepest descent method suitable for analyzing this type of situation. Before stating the theorem, we describe the set-up and make a number of assumptions.

As in Section \ref{similaritysec}, we assume the variables $\zeta$ and $k_0$ are given by
$$\zeta = x/t, \qquad k_0 = \sqrt{\zeta/12}.$$
Note that $0 < tk_0^3 < C$ in $\mathbb{S}$. Hence, as $(x,t) \to \infty$ within $\mathbb{S}$, the two stationary points $\pm k_0$ approach $0$ at least as fast as $t^{-1/3}$:
\begin{align}\label{k0merge}
0 < k_0 < C t^{-1/3} \quad \text{for} \quad (x,t) \in \mathbb{S}.
\end{align}
For each point $(x,t) \in \mathbb{S}$, we let $Y = Y(x,t) = k_0 Z$ denote the contour $Z$ defined in (\ref{Zdef}) scaled by the factor $k_0$ (see Figure \ref{Yfig}), i.e. $Y = Y_1 \cup Y_2 \cup Y_3$, where
\begin{align} \nonumber
&Y_1 = \bigl\{k_0 + ue^{\frac{i\pi}{6}}\, \big| \, 0 \leq u < \infty\bigr\} \cup \bigl\{-k_0 + ue^{\frac{5i\pi}{6}}\, \big| \, 0 \leq u < \infty\bigr\},  
	\\ \nonumber
&Y_2 = \bigl\{-k_0 + ue^{-\frac{5i\pi}{6}}\, \big| \, 0 \leq u < \infty\bigr\} \cup \bigl\{k_0 + ue^{-\frac{i\pi}{6}}\, \big| \, 0 \leq u < \infty\bigr\}, 
	\\ \label{Ydef}
& Y_3 = \bigl\{u\, \big| \, -k_0 \leq u \leq k_0 \bigr\}.
\end{align}
Let $\Gamma = \Gamma(x,t)$ be a family of oriented contours parametrized by $(x,t) \in \mathbb{S}$ and let $\hat{\Gamma} = \Gamma  \cup \{|k| = 1\}$ denote the union of $\Gamma$ with the unit circle oriented counterclockwise. Assume that, for each point $(x,t) \in \mathbb{S}$:
\begin{itemize}

\item[($\Gamma$1)] $\Gamma$ and $\hat{\Gamma}$ are Carleson jump contours up to reorientation of a subcontour.

\item[($\Gamma$2)] $\Gamma$ contains $Y$ as an oriented subcontour.

\item[($\Gamma$3)] $\Gamma$ is invariant as a set under the map $k \mapsto - \bar{k}$. Moreover, the orientation of $\Gamma$ is such that if $k$ traverses $\Gamma$ in the positive direction, then $-\bar{k}$ traverses $\Gamma$ in the negative direction.

\item[($\Gamma$4)] The point $\infty$ in the Riemann sphere can be approached nontangentially with respect to $\Gamma$.
\end{itemize}
Assume the operator $\mathcal{S}_{\hat{\Gamma}}$ defined in (\ref{Cauchysingulardef}) satisfies $\sup_{(x,t) \in \mathbb{S}} \|\mathcal{S}_{\hat{\Gamma}}\|_{\mathcal{B}(L^2(\hat{\Gamma}))} < \infty$.

Consider the following family of $L^2$-RH problems parametrized by the two parameters $(x,t) \in \mathbb{S}$:
\begin{align}\label{RHmIV}
\begin{cases} 
m(x, t, \cdot) \in I + \dot{E}^2(\C \setminus \Gamma), \\
m_+(x, t, k) = m_-(x, t, k) v(x, t, k) \quad \text{for a.e.} \ k \in \Gamma, 
\end{cases} 
\end{align}
where the jump matrix $v(x, t, k)$ satisfies
\begin{align}\label{winL1L2LinfIV}
&w(x, t,\cdot) := v(x, t,\cdot) - I \in L^1(\Gamma) \cap L^2(\Gamma) \cap L^\infty(\Gamma), \qquad (x,t) \in \mathbb{S},
	\\
&\det v(x, t, \cdot) = 1 \;\; \text{a.e. on $\Gamma$}, \qquad (x,t) \in \mathbb{S},
\end{align}
and
\begin{align}\label{vsymmIV}
  v(x,t,k) = \sigma_3 \overline{v(x, t, -\bar{k})} \sigma_3, \qquad (x,t) \in \mathbb{S}, \quad k \in \Gamma.
\end{align}

\begin{figure}
\begin{center}
\bigskip
 \begin{overpic}[width=.6\textwidth]{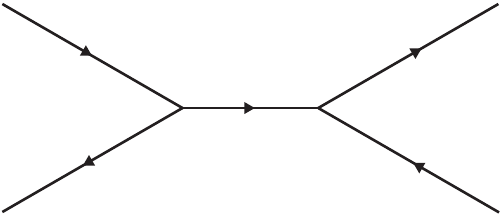}
 \put(78,35.5){$Y_1$}
 \put(17,35.5){$Y_1$}
 \put(17,5){$Y_2$}
 \put(78,5){$Y_2$}
 \put(47,25){$Y_3$}
 \put(61,17){$k_0$}
 \put(33,17){$-k_0$}
 \end{overpic}
   \bigskip
   \begin{figuretext}\label{Yfig}
      The contour $Y = Y_1 \cup Y_2 \cup Y_3$.
            \end{figuretext}
   \end{center}
\end{figure}
Suppose
\begin{align}\label{wL12inftyIV}
\begin{cases} 
\|w(x,t,\cdot)\|_{L^1(\Gamma \setminus Y)} = O(t^{-\frac{2}{3}}), 
	\\
\|w(x,t,\cdot)\|_{L^\infty(\Gamma \setminus Y)} = O(t^{-\frac{1}{3}}),
\end{cases} \qquad t \to \infty, \quad (x,t) \in \mathbb{S},	
\end{align}
uniformly with respect to $x$. Moreover, suppose that for $k \in Y$ the jump matrix $v$ has the form
\begin{align}\label{vjumpIV}
v(x,t,k) = \begin{cases}
\begin{pmatrix} 1 & 0 \\ R_1(x, t, k) e^{t\Phi(\zeta, k)}  & 1  \end{pmatrix}, & k \in Y_1,
	\\
\begin{pmatrix} 1 & R_2(x, t, k) e^{-t\Phi(\zeta, k)} \\ 0 & 1 \end{pmatrix}, & k \in Y_2, 
	\\
\begin{pmatrix} 1 & -R_4(x,t, k) e^{-t\Phi(\zeta, k)} \\ 0 & 1 \end{pmatrix}
\begin{pmatrix} 1 & 0 \\ R_3(x,t, k) e^{t\Phi(\zeta, k)}  & 1  \end{pmatrix}, &  k \in Y_3,
\end{cases}
\end{align}
where:
\begin{itemize}
\item The phase $\Phi(\zeta, k)$ is given by
\begin{align}\label{PhiIV}
\Phi(\zeta, k) = -24ik_0^2k + 8ik^3.
\end{align}

\item There exist constants $s \in i\R$, $|s| < 1$, and $L > 0$ such that the functions $\{R_j(x,t, k)\}_1^4$ satisfy the inequalities
\begin{align} \label{RinequalitiesIV}
\begin{cases}
  |R_1(x,t, k) - s| \leq L |k| e^{\frac{t}{2}|\re \Phi(\zeta, k)|}, & k \in Y_1, 	\\
   |R_2(x,t, k) + s| \leq L |k| e^{\frac{t}{2}|\re \Phi(\zeta, k)|}, & k \in Y_2,	\\
   |R_3(x,t, k) - s| \leq L |k|, & k \in Y_3,	\\
   |R_4(x,t, k) + s| \leq L |k|, & k \in Y_3,
\end{cases} 
\end{align}
for all $(x,t) \in \mathbb{S}$ with $t$ sufficiently large.
\end{itemize}

\begin{theorem}[Nonlinear steepest descent]\label{steepestdescentthIV}
Under the above assumptions, the $L^2$-RH problem (\ref{RHmIV}) has a unique solution for all sufficiently large $t$ and this solution satisfies
\begin{align}\label{limlm12IV}
\ntlim_{k\to \infty} (k m(x,t,k))_{12}
= \frac{u^P\big(\frac{-x}{(3t)^{1/3}}; s, 0,-s\big)}{2(3t)^{\frac{1}{3}}}  +  O\big(t^{-\frac{2}{3}}\big), \qquad t \to \infty, \quad (x,t) \in \mathbb{S},
\end{align}
where the error term is uniform with respect to $x$ and $u^P(\cdot; s, 0,-s)$ denotes the solution of the Painlev\'e II equation (\ref{painleve2}) corresponding to $(s_1, s_2, s_3) = (s,0,-s)$ according to (\ref{painlevebijection}).
\end{theorem}

\begin{remark}\upshape
The conclusion of Theorem \ref{steepestdescentthIV} can be stated more explicitly as follows: There exist constants $T> 0$ and $K>0$ such that the solution $m(x, t,k)$ of (\ref{RHmIV}) exists, the nontangential limit $\ntlim_{k\to \infty}(k m(\zeta,t,k))_{12}$ exists, and the inequality
\begin{align}\nonumber
  \left| \ntlim_{k\to \infty} (k m(\zeta,t,k))_{12} - \frac{u^P\big(\frac{-x}{(3t)^{1/3}}; s, 0,-s\big)}{2(3t)^{\frac{1}{3}}}\right|
  \leq K t^{-\frac{2}{3}}
\end{align}
holds for all $(x, t) \in \mathbb{S}$ such that $t > T$.
\end{remark}

\proofbegin
Since $\det v = 1$, the solution of (\ref{RHmIV}) is unique if it exists. Let  $s \in i\R$, $|s| < 1$, be the constant in (\ref{RinequalitiesIV}). We henceforth suppose $t$ is so large that $k_0 < 1/2$.
Let $m^Z(s,y,z, z_0)$ denote the solution of the RH problem (\ref{RHmZ}) whose existence is ascertained by Theorem \ref{mZth}. We note that if
$$y = \frac{-x}{(3t)^{1/3}}, \qquad z = (3t)^{\frac{1}{3}} k,$$
then
$$2i\bigg(yz + \frac{4z^3}{3}\bigg) = t\Phi(\zeta,k).$$
Hence we define $m_0(x,t,k)$ by
\begin{align}\label{m0def6}
m_0(x, t, k) = m^Z\biggl(s, \frac{-x}{(3t)^{\frac{1}{3}}}, (3t)^{\frac{1}{3}} k, (3t)^{\frac{1}{3}} k_0 \biggr), \qquad |k| \leq 1.
\end{align}
Since $m_0(x,t, \cdot) \in (\dot{E}^2 \cap E^\infty)(\{|k| < 1\} \setminus Y)$, it follows that $m$ satisfies the $L^2$-RH problem (\ref{RHmIV}) iff the function $\hat{m}(x, t, k)$ defined by
\begin{align*}
\hat{m}(x,t,k) = \begin{cases}
m(x, t, k)m_0(x,t,k)^{-1}, & |k| < 1,\\
m(x, t, k), & \text{otherwise},
\end{cases}
\end{align*}
satisfies the $L^2$-RH problem
\begin{align}\label{RHmhatIV}
\begin{cases}
\hat{m}(x,t, \cdot) \in I + \dot{E}^2(\C \setminus \hat{\Gamma}), \\
\hat{m}_+(x,t,k) = \hat{m}_-(x, t, k) \hat{v}(x, t, k)  \quad \text{for a.e.} \ k \in \hat{\Gamma},
\end{cases}
\end{align}
where the jump matrix $\hat{v}$ is given by
\begin{align*}
\hat{v}(x, t, k) 
=  \begin{cases}
 m_{0-}(x, t, k) v(x, t, k) m_{0+}(x,t,k)^{-1}, & k \in \Gamma \cap \{|k| < 1\}, \\
m_0(x, t, k)^{-1}, & |k| = 1, \\
v(x, t, k),  & \text{otherwise}.
\end{cases}
\end{align*}
Let us write $Y_1 = Y_1^+ \cup Y_1^-$, where $Y_1^+$ and $Y_1^-$ denote the parts of $Y_1$ in the right and left half-planes, respectively. Similarly, we write $Y_2 = Y_2^+ \cup Y_2^-$, see Figure \ref{Yfig2}.

\begin{figure}
\begin{center}
\bigskip
 \begin{overpic}[width=.6\textwidth]{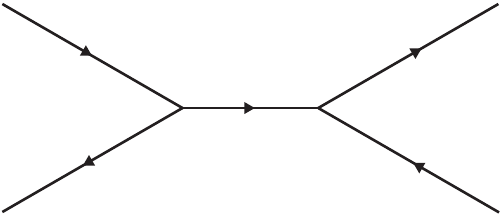}
 \put(78,35.5){$Y_1^+$}
 \put(17,35.5){$Y_1^-$}
 \put(17,5){$Y_2^-$}
 \put(78,5){$Y_2^+$}
 \put(47,25){$Y_3$}
 \put(61,17){$k_0$}
 \put(33,17){$-k_0$}
 \end{overpic}
   \bigskip
   \begin{figuretext}\label{Yfig2}
     $Y_j^+$ and $Y_j^-$ denote the parts of $Y_j$, $j = 1,2$, in the right and left half-planes, respectively. 
            \end{figuretext}
   \end{center}
\end{figure}

\begin{claim}\label{claim1}
The function $\hat{w} = \hat{v} - I$ satisfies
\begin{align}\label{hatwestimateIV}
\hat{w}(x,t,k) = \begin{cases} O\bigl(|k| e^{-4 t |k - k_0|^3}\bigr), & k \in Y_1^+ \cup Y_2^+, \\
O\bigl(|k| e^{-4 t |k + k_0|^3}\bigr), & k \in Y_1^- \cup Y_2^-, \\
O(|k|), & k \in Y_3, 
\end{cases}
\qquad t \to \infty, \quad (x,t) \in \mathbb{S}, 
\end{align}
where the error term is uniform with respect to $x$ and $k$ in the given ranges.
\end{claim}

{\it Proof of Claim \ref{claim1}.}
We give the proof for $k \in Y_1^+ \cup Y_2^+ \cup Y_3$; the proof for $k \in Y_1^- \cup Y_2^-$ is similar.

The definition (\ref{PhiIV}) of $\Phi$ implies 
$$\re \Phi(\zeta, k) = \mp (8r^3 + 12\sqrt{3}k_0r^2), \qquad k = k_0 + r e^{\pm \frac{i\pi}{6}}, \quad r > 0.$$
Hence
\begin{align} \label{RePhiZ}
\begin{cases}
 \re \Phi(\zeta,k) \leq -8 |k - k_0|^3, \qquad &  k \in Y_1^+,
    	\\ 
 \re \Phi(\zeta,k) \geq 8 |k - k_0|^3, & k \in Y_2^+,
\end{cases} \quad (x,t) \in \mathbb{S}.
\end{align}
If $k \in Y_1^+$ and $|k| > 1$, then only the $(21)$ entry of $\hat{w}$ is nonzero and equations (\ref{RinequalitiesIV}) and (\ref{RePhiZ}) yield
\begin{align} \nonumber
& |\hat{w}_{21}(x,t,k)| = |R_1(x,t, k) | e^{-t|\re \Phi(\zeta, k)|} \leq (L |k|e^{\frac{t}{2}|\re \Phi(\zeta, k)|} + |s|)e^{-t|\re \Phi(\zeta, k)|}  
	\\
&\leq (L |k| + |s|)e^{-4 t |k-k_0|^3}  \leq C|k| e^{-4 t |k-k_0|^3}, \qquad (x,t) \in \mathbb{S}, \quad  k \in Y_1, \quad |k| > 1.
\end{align}
A similar argument applies when $k \in Y_2^+$. This proves (\ref{hatwestimateIV}) for $k \in Y_1^+ \cup Y_2^+$ with $|k| > 1$. 

It remains to consider the case $|k| < 1$. In this case, we have
$$\hat{w}(x, t,k)
= m_{0-}(x, t, k)  u(x,t,k) m_{0+}(x,t,k)^{-1}, \qquad k \in Y, \quad |k| < 1,$$
where
$$u(x,t,k) := v(x, t, k) - v^Z\biggl(s, \frac{-x}{(3t)^{\frac{1}{3}}}, (3t)^{\frac{1}{3}} k, (3t)^{\frac{1}{3}} k_0\biggr).$$
In view of equation (\ref{mZbounded}), $m_{0+}(x, t, k)$ and $m_{0-}(x, t, k)$ remain uniformly bounded for $(x,t) \in \mathbb{S}$ and $k \in Y$. Therefore, it is enough to prove that 
\begin{align}\label{uboundIV}
u(x,t,k) = \begin{cases}
O\bigl(|k| e^{-4 t |k - k_0|^3}\bigr), & k \in Y_1^+ \cup Y_2^+,
	\\
O(|k|), & k \in Y_3,	
\end{cases}\quad
 t \to \infty, \quad (x,t) \in \mathbb{S}, \quad |k| < 1, \end{align}
uniformly with respect to $x$ and $k$. Equation (\ref{uboundIV}) follows from the assumptions (\ref{vjumpIV})-(\ref{RinequalitiesIV}) on $v$. Indeed, for $k \in Y_1^+$, only the $(21)$ entry of $u(x,t,k)$ is nonzero and equations (\ref{RinequalitiesIV}) and (\ref{RePhiZ}) yield
\begin{align} \nonumber
|(u(x, t, k))_{21}| = &\; |R_1(x,t, k) -s| e^{-t|\re \Phi(\zeta, k)|}
\leq L |k| e^{-4 t |k-k_0|^3},
	\\ \label{ueq1}
& \hspace{4cm}  (x,t) \in \mathbb{S}, \quad k \in Y_1^+, \quad |k| < 1.
\end{align}
This shows (\ref{uboundIV}) for $k \in Y_1^+$ with $|k| < 1$. The proof for $k \in Y_2^+$ with $|k| < 1$  is similar. 

For $k \in Y_3$, we have $|e^{\pm t\Phi(\zeta, k)} | = 1$ and
$$u(x,t,k)
= \begin{pmatrix} - R_3(x,t,k)R_4(x,t,k) - s^2 & -(R_4(x,t, k) + s) e^{-t\Phi(\zeta, k)} \\ 
(R_3(x, t, k) - s) e^{t\Phi(\zeta, k)} & 0 \end{pmatrix}.$$
Hence, by (\ref{RinequalitiesIV}),
\begin{align*}
& |(u(x,t,k))_{11}| \leq |R_3 - s| |R_4| + |s| |R_4 +s| \leq C |k|,
	\\
& |(u(x,t,k))_{21}| \leq |R_3 - s| \leq L |k|, \qquad |(u(x,t,k))_{12}| \leq |R_4 + s| \leq L |k|,
\end{align*}
for $k\in Y_3$ and $(x,t) \in \mathbb{S}$. This shows (\ref{uboundIV}) also for $k \in Y_3$ and completes the proof of the claim.
\proofendcontinue

\begin{claim}\label{claim2}
We have
\begin{align}\label{hatwLinftyIV}
& \|\hat{w}(x, t, \cdot)\|_{L^\infty(\hat{\Gamma})} = O\big(t^{-\frac{1}{3}}\big), && t \to \infty, \quad (x,t) \in \mathbb{S},
	\\\label{hatwL2IV}
& \|\hat{w}(x, t, \cdot)\|_{L^2(\hat{\Gamma})} = O\big(t^{-\frac{1}{3}}\big), && t \to \infty, \quad (x,t) \in \mathbb{S},
\end{align}
and
\begin{align}\label{hatwL1IV}
\|\hat{w}(x, t, \cdot)\|_{L^1(Y)} =  O\big(t^{-\frac{2}{3}}\big), \qquad t \to \infty, \quad (x,t) \in \mathbb{S},
\end{align}
where the error terms are uniform with respect to $x$.
\end{claim}

{\it Proof of Claim \ref{claim2}.}
Let $p = 2$ or $p=\infty$. Then
\begin{align}\label{whatnormIV}
\|\hat{w}(x, t, \cdot)\|_{L^p(\hat{\Gamma})} 
\leq &\; \|\hat{w}(x, t, \cdot)\|_{L^p(\Gamma \setminus Y)} + \|m_0(x, t, \cdot)^{-1} - I \|_{L^p(|k| = 1)} 
+ \|\hat{w}(x, t, \cdot)\|_{L^p(Y)}.
\end{align}
On $\Gamma\setminus Y$, the matrix $\hat{w}$ is given by either $w$ or $m_0wm_0^{-1}$.
Hence $\|\hat{w}(x, t, \cdot)\|_{L^\infty(\Gamma \setminus Y)} = O(t^{-\frac{1}{3}})$ and 
$$\|\hat{w}(x, t, \cdot)\|_{L^2(\Gamma \setminus Y)} \leq \sqrt{\|\hat{w}(x, t, \cdot)\|_{L^\infty(\Gamma \setminus Y)} \|\hat{w}(x, t, \cdot)\|_{L^1(\Gamma \setminus Y)}} = O(t^{-\frac{1}{2}})$$ 
by the assumption (\ref{wL12inftyIV}) and the boundedness of $m_0$.
Moreover, by (\ref{mZasymptotics}), 
\begin{align}\label{m0invminusI}
&\|m_0(x, t, k)^{-1} - I \|_{L^p(|k| = 1)}   = 
\biggl \|m^Z\biggl(s, \frac{-x}{(3t)^{\frac{1}{3}}}, (3t)^{\frac{1}{3}} k, (3t)^{\frac{1}{3}} k_0\biggr)^{-1}  - I \biggr\|_{L^p(|k| = 1)}  = O(t^{-\frac{1}{3}}), 
\end{align}
uniformly in $x$. 
By (\ref{k0merge}), we can choose a constant $K > 0$ such that $2 k_0 \leq K t^{-1/3}$ for all $(x,t) \in \mathbb{S}$. 
Then, for $|k| > K t^{-1/3}$, we have
\begin{align}\label{hatwKest1}
  |\hat{w}(x,t,k)| & \leq \frac{C}{4t|k|^2} 
  \leq C t^{-\frac{1}{3}},  \qquad  |k| > K t^{-\frac{1}{3}}, \quad k \in Y,
\end{align}
because equation (\ref{hatwestimateIV}) implies
\begin{align*}
  |\hat{w}(x,t,k)| & \leq C |k| e^{-4 t |k - k_0|^3} 
  \leq C |k| \frac{4t|k-k_0|^3 e^{-4t|k - k_0|^3}}{4t|k-k_0|^3} \leq \frac{C|k|}{4t|k - k_0|^3} 
  	\\
&  \leq \frac{C}{4t|k|^2} 
  \leq C t^{-\frac{1}{3}},  \qquad  |k| > K t^{-\frac{1}{3}}, \quad k \in Y_1^+ \cup Y_2^+,
\end{align*}
and a similar argument applies for $k \in Y_1^- \cup Y_2^-$.
On the other hand, if $|k| \leq K t^{-1/3}$, equation (\ref{hatwestimateIV}) implies
\begin{align}\label{hatwKest2}
|\hat{w}(x,t,k)| \leq C |k| \leq C t^{-\frac{1}{3}}, \qquad |k| \leq K t^{-\frac{1}{3}}, \quad k \in Y.
\end{align}
Using equations (\ref{hatwKest1}) and (\ref{hatwKest2}), we can estimate the third term on the right-hand side of (\ref{whatnormIV}) as follows
\begin{align*}
 \|\hat{w}(x, t, \cdot)\|_{L^\infty(Y)} & = O(t^{-\frac{1}{3}}),
	\\
 \|\hat{w}(x, t, \cdot)\|_{L^2(Y)}
&  \leq C\sqrt{\int_{Y \cap \{|k| > K t^{-\frac{1}{3}}\}}\Big( \frac{1}{t|k|^2}\Big)^2 d|k| + \int_{Y \cap \{|k| \leq K t^{-\frac{1}{3}}\}} |k|^2 d|k|}
	\\
& \leq C \sqrt{ \frac{1}{t^2} \int_{t^{-\frac{1}{3}}}^\infty u^{-4} du + \int_0^{t^{-\frac{1}{3}}}  u^2 du}
= O(t^{-\frac{1}{2}}).
\end{align*}
This proves (\ref{hatwLinftyIV}) and (\ref{hatwL2IV}).

In order to prove (\ref{hatwL1IV}), we note that (\ref{hatwestimateIV}) implies 
\begin{align}\nonumber
  \|\hat{w}(x, t, \cdot)\|_{L^1(Y)} 
\leq &\; C \int_{Y_1 \cup Y_2} |\hat{w}(x,t,k)| |dk|  + C \int_{Y_3} |\hat{w}(x,t,k)| |dk|
	\\\nonumber
\leq  &\; C \int_0^\infty (u + k_0) e^{- 4 t u^3} du + C \int_{-k_0}^{k_0} |u| du
	\\	 \label{hatwL1estimateIV}
=  &\; 
  C \frac{\Gamma(\frac{2}{3})}{6 2^{1/3}}t^{-\frac{2}{3}}
+ C \frac{k_0 \Gamma (\frac{4}{3})}{2^{2/3}}t^{-\frac{1}{3}}
    + C k_0^2
\leq C t^{-\frac{2}{3}}, \qquad (x,t) \in \mathbb{S}.
\end{align}
This proves (\ref{hatwL1IV}).
\proofendcontinue

Let $\hat{\mathcal{C}}$ denote the Cauchy operator associated with $\hat{\Gamma}$:
$$(\hat{\mathcal{C}} f)(z) = \frac{1}{2\pi i} \int_{\hat{\Gamma}} \frac{f(s)}{s - z} ds, \qquad z \in \C \setminus \hat{\Gamma}.$$
We denote the nontangential boundary values of $\hat{\mathcal{C}}f$ from the left and right sides of $\hat{\Gamma}$ by $\hat{\mathcal{C}}_+ f$ and $\hat{\mathcal{C}}_-f$ respectively. 
We define $\hat{\mathcal{C}}_{\hat{w}}: L^2(\hat{\Gamma}) + L^\infty(\hat{\Gamma}) \to L^2(\hat{\Gamma})$ by $\hat{\mathcal{C}}_{\hat{w}}f = \hat{\mathcal{C}}_-(f \hat{w})$.

\begin{claim}\label{claim3}
There exists a $T > 0$ such that $I - \hat{\mathcal{C}}_{\hat{w}(x, t, \cdot)} \in \mathcal{B}(L^2(\hat{\Gamma}))$ is invertible for all $(x,t) \in \mathbb{S}$ with $t > T$.
\end{claim}

{\it Proof of Claim \ref{claim3}.}
The assumption that $\mathcal{S}_{\hat{\Gamma}}$ is uniformly bounded on $L^2(\hat{\Gamma})$ together with the Sokhotski-Plemelj formula $\hat{\mathcal{C}}_- = \frac{1}{2}(-I + \mathcal{S}_{\hat{\Gamma}})$ show that $\sup_{(x,t) \in \mathbb{S}} \|\hat{\mathcal{C}}_-\|_{\mathcal{B}(L^2(\hat{\Gamma}))} < \infty$. 
Thus, by (\ref{hatwLinftyIV}),
\begin{align}\label{ChatwnormIV}
\|\hat{\mathcal{C}}_{\hat{w}}\|_{\mathcal{B}(L^2(\hat{\Gamma}))} \leq C \|\hat{w}\|_{L^\infty(\hat{\Gamma})}  
= O(t^{-\frac{1}{3}}), \qquad t \to \infty, \quad (x,t) \in \mathbb{S},
\end{align}
uniformly with respect to $x$. This proves the claim. 
\proofendcontinue

In view of Claim \ref{claim3}, we may define the $2 \times 2$-matrix valued function $\hat{\mu}(x, t, k)$ whenever $t > T$ by
\begin{align}\label{hatmudefIV}
\hat{\mu} = I + (I - \hat{\mathcal{C}}_{\hat{w}})^{-1}\hat{\mathcal{C}}_{\hat{w}}I  \in I + L^2(\hat{\Gamma}).
\end{align}

\begin{claim}\label{claim4}
The function $\hat{\mu}(x, t, k)$ satisfies
\begin{align}\label{muhatestimateIV}
\|\hat{\mu}(x,t,\cdot) - I\|_{L^2(\hat{\Gamma})} = O\big(t^{-\frac{1}{3}}\big), \qquad t \to \infty, \quad (x,t) \in \mathbb{S},
\end{align}
where the error term is uniform with respect to $x$.
\end{claim}

{\it Proof of Claim \ref{claim4}.}
We have
\begin{align*}
\|(I - \hat{\mathcal{C}}_{\hat{w}})^{-1}\|_{\mathcal{B}(L^2(\hat{\Gamma}))} 
& \leq \sum_{j=0}^\infty \|\hat{\mathcal{C}}_{\hat{w}}\|_{\mathcal{B}(L^2(\hat{\Gamma}))}^j
 = \frac{1}{1 - \|\hat{\mathcal{C}}_{\hat{w}}\|_{\mathcal{B}(L^2(\hat{\Gamma}))}},
\end{align*}
whenever $\|\hat{\mathcal{C}}_{\hat{w}}\|_{\mathcal{B}(L^2(\hat{\Gamma}))} < 1$.
Using the bound $\sup_{\zeta \in \mathcal{I}} \|\hat{\mathcal{C}}_-\|_{\mathcal{B}(L^2(\hat{\Gamma}))} < \infty$ and equation (\ref{ChatwnormIV}), we find
\begin{align*}
\|\hat{\mu} - I\|_{L^2(\hat{\Gamma})} & = 
\|(I- \hat{\mathcal{C}}_{\hat{w}})^{-1}\hat{\mathcal{C}}_{\hat{w}}I\|_{L^2(\hat{\Gamma})} 
 \leq \|(I - \hat{\mathcal{C}}_{\hat{w}})^{-1}\|_{\mathcal{B}(L^2(\hat{\Gamma}))}
\|\hat{\mathcal{C}}_-(\hat{w})\|_{L^2(\hat{\Gamma})}
	\\
& \leq  
\frac{C\|\hat{w}\|_{L^2(\hat{\Gamma})}}{1 - \|\hat{\mathcal{C}}_{\hat{w}}\|_{\mathcal{B}(L^2(\hat{\Gamma}))}}
\leq C\|\hat{w}\|_{L^2(\hat{\Gamma})}
\end{align*}
for all $(x,t) \in \mathbb{S}$ with $t$ large enough. In view of (\ref{hatwL2IV}), this gives (\ref{muhatestimateIV}).
\proofendcontinue

\begin{claim}\label{claim5}
There exists a unique solution $\hat{m} \in I + \dot{E}^2(\C \setminus \hat{\Gamma})$ of the $L^2$-RH problem (\ref{RHmhatIV}) whenever $t > T$. This solution is given by
\begin{align}\label{hatmrepresentationIV}
\hat{m}(x, t, k) = I + \hat{\mathcal{C}}(\hat{\mu}\hat{w})  = I + \frac{1}{2\pi i}\int_{\hat{\Gamma}} \hat{\mu}(x, t, s) \hat{w}(x, t, s) \frac{ds}{s - k}.
\end{align}
\end{claim}

{\it Proof of Claim \ref{claim5}.}
Uniqueness follows since $\det \hat{v} = 1$. Moreover, equation (\ref{hatmudefIV}) implies that $\hat{\mu} - I = \hat{\mathcal{C}}_{\hat{w}} \hat{\mu}$. Hence, by the theory of $L^2$-RH problems, $\hat{m} = I + \hat{\mathcal{C}}(\hat{\mu} \hat{w})$ satisfies the $L^2$-RH problem (\ref{RHmhatIV}). 
\proofendcontinue

\begin{claim}\label{claim6}
For each point $(x, t) \in \mathbb{S}$ with $t > T$, the nontangential limit of $k(\hat{m}(x,t,k) - I)$ as $k \to \infty$ exists and is given by
\begin{align*}
\ntlim_{k\to \infty} k(\hat{m}(x,t,k) - I) 
= -\frac{1}{2\pi i}\int_{\hat{\Gamma}} \hat{\mu}(x,t,k) \hat{w}(x,t,k) dk.
\end{align*}
\end{claim}

{\it Proof of Claim \ref{claim6}.}
Fix $(x,t) \in \mathbb{S}$ and let $W_{a,b} = \{a \leq \arg k \leq b\}$ be a nontangential sector at $\infty$ with respect to $\hat{\Gamma}$. This means that there exists a $\delta > 0$ such that $W_{a - \delta, b + \delta}$ does not intersect $\hat{\Gamma} \cap \{|z| > R\}$ whenever $R>0$ is large enough. It follows that there exists a $c>0$ such that $|s-k| > c(|s| + |k|)$ for all $s \in \hat{\Gamma}$ and all $k \in W_{a,b}$ with $|k|$ sufficiently large.
Thus, since $\hat{\mu}\hat{w} \in L^1(\hat{\Gamma})$, dominated convergence implies
\begin{align*}
& \ntlim_{k\to \infty} \bigg|\int_{\hat{\Gamma}} (\hat{\mu} \hat{w})(x, t, s) \frac{k ds}{s - k} + \int_{\hat{\Gamma}} (\hat{\mu} \hat{w})(x, t, s) ds\bigg| 
\leq \ntlim_{k\to \infty} \int_{\hat{\Gamma}} |(\hat{\mu} \hat{w})(x, t, s)| \frac{|s| |ds|}{|s - k|} 
	\\
& 
\leq \ntlim_{k\to \infty} \int_{\hat{\Gamma}} |(\hat{\mu} \hat{w})(x, t, s)| \frac{|s| |ds|}{c(|s| + |k|)} 
= 0.
\end{align*}
The claim now follows from (\ref{hatmrepresentationIV}).
\proofendcontinue

Claim \ref{claim6} implies 
\begin{align}\nonumber
& \ntlim_{k\to \infty} k(m(x,t,k) - I) = \ntlim_{k\to \infty} k(\hat{m}(x,t,k) - I) 
	\\ \label{limlmminusIIV}
= & - \frac{1}{2\pi i} \int_{|k| = 1} \hat{\mu}(x,t,k) (m_0(x,t,k)^{-1} - I)dk
 - \frac{1}{2\pi i}\int_{\Gamma} \hat{\mu}(x,t,k) \hat{w}(x,t,k) dk.
\end{align}
By (\ref{mZasymptotics}),
\begin{align} \nonumber
  m_0(x, t, k)^{-1} = &\; m^Z\biggl(s, \frac{-x}{(3t)^{\frac{1}{3}}}, (3t)^{\frac{1}{3}}k, (3t)^{\frac{1}{3}}k_0 \biggr)^{-1}
  	\\ \label{m0invasymptoticsIV}
= &\; I - \frac{B(x, t)}{(3t)^{\frac{1}{3}}k} + O\biggl(\frac{1}{(3t)^{\frac{2}{3}}k^2}\biggr), \qquad t^{\frac{1}{3}} |k| \to \infty,
\end{align}
where $B(x, t)$ is defined by
\begin{align*}
B(x, t) =  \frac{1}{2}\begin{pmatrix} 0 & u\Big(\frac{-x}{(3t)^{\frac{1}{3}}}; s, 0, -s\Big) \\ u\Big(\frac{-x}{(3t)^{\frac{1}{3}}}; s, 0, -s\Big) & 0 \end{pmatrix}.
\end{align*}
Using (\ref{m0invminusI}), (\ref{muhatestimateIV}), and (\ref{m0invasymptoticsIV}), we find
\begin{align}\nonumber
 &\int_{|k| = 1}  \hat{\mu}(x,t,k) (m_0(x,t,k)^{-1} - I) dk
	\\\nonumber
= &\int_{|k| = 1} (m_0(x,t,k)^{-1} - I) dk  + \int_{|k| = 1}  (\hat{\mu}(x,t,k) - I) (m_0(x,t,k)^{-1} - I) dk
 	\\\nonumber
= & -\frac{B(x, t)}{(3t)^{\frac{1}{3}}} \int_{|k| = 1} \frac{dk}{k } 
 + O\big(t^{-\frac{2}{3}}\big) 
  + O\Big(\|\hat{\mu}(x,t,\cdot) - I\|_{L^2(\hat{\Gamma})} \|(m_0(x,t,\cdot)^{-1} - I)\|_{L^2(|k| = 1)} \Big)
	\\ \label{inthatmumjIV}
= & - \frac{2\pi i B(x, t)}{(3t)^{\frac{1}{3}}} + O\big(t^{-\frac{2}{3}}\big), \qquad t \to \infty, \quad (x,t) \in \mathbb{S},
\end{align}
uniformly with respect to $x$. 
On the other hand, 
\begin{align}\nonumber
\biggl|\int_{\Gamma}\hat{\mu}(\zeta,t,k) \hat{w}(\zeta,t,k) dk\biggr|
 \leq \|\hat{\mu} - I\|_{L^2(\Gamma)}  \|\hat{w}\|_{L^2(\Gamma)} + \|\hat{w}\|_{L^1(\Gamma)}.
\end{align}
By (\ref{hatwL2IV}) and (\ref{muhatestimateIV}), both $\|\hat{w}\|_{L^2(\Gamma)}$ and $\|\hat{\mu} - I\|_{L^2(\Gamma)}$ are $O\big(t^{-\frac{1}{3}}\big)$ as $t \to \infty$.
Moreover, the $L^1$ norm of $w$ is $O(t^{-\frac{2}{3}})$ on $\Gamma \setminus Y$ by (\ref{wL12inftyIV}). 
Since equations (\ref{m0def6}) and (\ref{mZbounded}) imply that $m_{0+}(x, t, k)$ and $m_{0-}(x, t, k)$ are uniformly bounded for $(x,t) \in \mathbb{S}$ and $k \in \Gamma \cap \{|k| < 1\}$, it follows that the $L^1$ norm of $\hat{w}$ also is $O(t^{-\frac{2}{3}})$ on $\Gamma \setminus Y$. On the other hand, the $L^1$ norm of $\hat{w}$ is $O(t^{-\frac{2}{3}})$ on $Y$ by (\ref{hatwL1IV}). Thus $\|\hat{w}\|_{L^1(\Gamma)} = O(t^{-\frac{2}{3}})$. We infer that
\begin{align} \label{intSigmahatmuhatwIV}
& \biggl|\int_{\Gamma} \hat{\mu}(x,t,l) \hat{w}(x,t,l) dl\biggr| = O\big(t^{-\frac{2}{3}}\big), \qquad t \to \infty, \quad (x,t) \in \mathbb{S},
\end{align}
uniformly with respect to $x$.
Equations (\ref{limlmminusIIV}), (\ref{inthatmumjIV}), and (\ref{intSigmahatmuhatwIV}) give (\ref{limlm12IV}).
\proofend

\section{Asymptotics in the self-similar sector}\nequation\label{selfsimilarsec}
The purpose of this section is to prove Theorem \ref{mainth2} which describes the asymptotic behavior of quarter plane solutions of the mKdV equation in the self-similar sector (\ref{sectorIVdef}). The proof relies on an application of the nonlinear steepest descent result of Theorem \ref{steepestdescentthIV}.

\begin{theorem}[Asymptotics in the self-similar sector]\label{mainth2}
Let $r \in C^{11}(\R)$ and $h \in C^5(\bar{D}_2)$ be two functions which satisfy the assumptions of Theorem \ref{mainth1}.
Given $N > 0$, let $\mathbb{S} = \{t>1, \, 0 < x < N t^{1/3}\}$ be the self-similar sector defined in (\ref{sectorIVdef}). Let $s = ir(0)$.

Then there exists a $T>0$ such that the $L^2$-RH problem (\ref{RHM}) has a unique solution and the limit in (\ref{ulim}) exists for all $(x,t) \in \mathbb{S}$ with $t > T$.
Moreover, the function $u(x,t)$ defined by (\ref{ulim}) satisfies
\begin{align}\label{ufinalIV}
u(x,t) = \frac{u^P\big(\frac{-x}{(3t)^{1/3}}; s, 0, -s\big)}{(3t)^{\frac{1}{3}}} 
+ O\bigl(t^{-\frac{2}{3}}\bigr),\qquad
t \to \infty, \quad (x,t) \in \mathbb{S},
\end{align}
where the error term is uniform with respect to $x$, and $u^P(\cdot; s, 0,-s)$ denotes the solution of the Painlev\'e II equation (\ref{painleve2}) corresponding to $(s_1, s_2, s_3) = (s,0,-s)$ according to (\ref{painlevebijection}). 
\end{theorem}

Before presenting the proof of Theorem \ref{mainth2}, we give a corollary.

Theorem \ref{mainth2} applies whenever the functions $r(k)$  and $h(k)$ satisfy the stated assumptions, regardless of whether these functions derive from a quarter plane solution or not. If we do assume that $r(k)$  and $h(k)$ are defined in terms of some initial and boundary values $u_0(x)$ and $\{g_j(t)\}_0^2$ via (\ref{hrdef}), we obtain the following corollary.

\begin{corollary}[Asymptotics for quarter plane solutions]\label{selfsimilarcorollary}
Suppose $u_0, g_0, g_1, g_2$ satisfy (\ref{ugjassump}) with $n = 11$ and $m = 4$, that the associated spectral functions satisfy the global relation (\ref{GR}), and that the homogeneous $L^2$-RH problem associated with (\ref{RHM}) has only the trivial solution for each $(x,t) \in [0,\infty) \times [0,\infty)$. Then the function $u(x,t)$  defined by (\ref{ulim}) is a classical solution of (\ref{mkdv}) in the quarter plane with initial data $u(x,0) = u_0(x)$  and boundary values $\partial_x^ju(0,t) = g_j(t)$, $j = 0,1,2$. Moreover, in the self-similar sector, the solution $u(x,t)$ satisfies
$$u(x,t) = O\bigl(t^{-\frac{2}{3}}\bigr),\qquad
t \to \infty, \quad (x,t) \in \mathbb{S},$$
uniformly with respect to $x$. 
\end{corollary}

\medskip\noindent
{\it Proof of Corollary \ref{selfsimilarcorollary}.}
Combining Theorem \ref{existenceth}  and Theorem \ref{mainth2}, we infer the existence of a classical quarter plane solution $u(x,t)$ which satisfies the asymptotic formula (\ref{ufinalIV}) in the self-similar sector. The formula (\ref{ufinalIV}), being uniformly valid as $x \to 0$, can only be consistent with the assumed decay of $u(0,t)$ provided that the coefficient $u^P$ vanishes at $x = 0$. In fact, it turns out that the global relation forces the parameter $s$ to vanish, which in turn implies that $u^P$ vanishes identically. 
To see this, we first note that the symmetries 
$$a(k) = \overline{a(-\bar{k})}, \quad b(k) = \overline{b(-\bar{k})}, \quad A(k) = \overline{A(-\bar{k})}, \quad B(k) = \overline{B(-\bar{k})},$$
imply that the four values $a(0)$, $b(0)$, $A(0)$, $B(0)$ are real. The global relation (\ref{GR}) and the unit determinant relation $a(k) \overline{a(\bar{k})} - b(k) \overline{b(\bar{k})} = 1$ then show that $h(0) = - b(0)/a(0)$, which gives $s = i r(0) = 0$.
For $s = 0$, the jump matrices in the Painlev\'e II RH problem (\ref{RHmP}) reduce to the identity matrix, so the solution $u^P(y; 0,0,0)$ corresponding to $s = 0$ is identically zero. This completes the proof of the corollary.
\proofend

\medskip\noindent
{\it Proof of Theorem \ref{mainth2}.}
Just like the proof of Theorem \ref{mainth1}, the proof consists of seven steps.

\medskip
{\bf Step 1: Introduce an analytic approximation of $h(k)$.}
Lemma \ref{decompositionlemma} applies also for $(x,t) \in \mathbb{S}$. Thus we find a decomposition 
$$h(k) = h_{a}(t, k) + h_{r}(t, k), \qquad k \in \partial D_1,$$
where $h_{a}$ and $h_{r}$ have the properties listed in Lemma \ref{decompositionlemma}.

\medskip
{\bf Step 2: Deform.}
Proceeding as in the proof of Theorem \ref{mainth1}, we note that $M$ satisfies the $L^2$-RH problem (\ref{RHM}) iff the function $M^h(x,t,k)$ defined by (\ref{checkMdef}) satisfies the $L^2$-RH problem (\ref{RHMcheck}).

\medskip
{\bf Step 3: Conjugate.}
As in the proof of Theorem \ref{mainth1}, we conjugate the RH problem to arrive at an appropriate factorization. However, instead of introducing $\delta(\zeta, k)$ by (\ref{deltadef}), we now define a function $\delta(k)$ by the equation obtained from (\ref{deltadef}) by replacing $k_0$ with  zero (recall that $k_0 \to 0$ in the self-similar sector). 
Thus, let
$$\Delta(k) = \begin{pmatrix} \delta(k)^{-1} & 0 \\  0 & \delta(k)   \end{pmatrix},$$
where
\begin{align}\label{deltadefIV}
 \delta(k) = e^{-\frac{1}{2\pi i} \int_\R \ln(1- |r(s)|^2) \frac{ds}{s - k}}, \qquad  k \in \C \setminus \R.
\end{align}
The function $\delta$ satisfies the following jump condition across the real axis:
\begin{align}\label{deltajumpIV}
 \delta_+(k) = \frac{\delta_-(k)}{1 - |r(k)|^2}, \qquad k \in \R.
\end{align}
The symmetry $r(k) = \overline{r(-\bar{k})}$ implies that $\Delta$ obeys the symmetries (\ref{Deltasymm}). 

\begin{lemma}\label{DeltalemmaIV}
The $2 \times 2$-matrix valued function $\Delta(k)$ satisfies
$$\Delta, \Delta^{-1} \in I + \dot{E}^2(\C \setminus \R) \cap E^\infty(\C \setminus \R).$$
\end{lemma}
\proofbegin The proof is similar to, but easier than, the proof of Lemma \ref{Deltalemma}. \proofendcontinue

By Lemma \ref{DeltalemmaIV}, $M^h(x,t,k)$ satisfies the $L^2$-RH problem (\ref{RHMcheck})
iff the function $\tilde{M}$ defined by
\begin{align}\label{MtildedefIV}
\tilde{M}(x,t,k) = e^{-\frac{i\pi}{4}\sigma_3}M^h(x,t,k)\Delta(k)e^{\frac{i\pi}{4}\sigma_3}
\end{align}
satisfies the $L^2$-RH problem
\begin{align}\label{RHMtildeIV}
\begin{cases}
\tilde{M}(x, t, \cdot)  \in I + \dot{E}^2(\C \setminus \Sigma),\\
\tilde{M}_+(x,t,k) = \tilde{M}_-(x, t, k) \tilde{J}(x, t, k)  \quad \text{for a.e.} \ k \in \Sigma,
\end{cases}
\end{align}
where
\begin{align} \nonumber
\tilde{J}&(x,t,k) = e^{-\frac{i\pi}{4}\sigma_3} \Delta_-^{-1}(k) J^h(x,t,k) \Delta_+(k) e^{\frac{i\pi}{4}\sigma_3}
	\\ \label{tildeJdefIV}
& =\begin{cases}
 \begin{pmatrix} 1 & 0 \\ i\delta(k)^{-2} h_{r}(t, k) e^{t\Phi(\zeta,k)} & 1 \end{pmatrix}, \qquad k \in \partial D_1,
 	\\
\begin{pmatrix} 1 - |r(k)|^2	& i\delta_-(k)^2 \frac{\overline{r(\bar{k})}}{1 - |r(k)|^2} e^{-t\Phi(\zeta,k)}	\\
 i\delta_+(k)^{-2} \frac{r(k)}{1 - |r(k)|^2} e^{t\Phi(\zeta,k)}	& 1 \end{pmatrix},  \qquad k \in \R,
	\\
\begin{pmatrix} 1 & -i\delta(k)^2 \overline{h_{r}(t,\bar{k})} e^{-t\Phi(\zeta,k)} \\ 0 & 1 \end{pmatrix}, \qquad k \in \partial D_4,
  \end{cases}
\end{align}
and we recall that $\Sigma = \partial D_1 \cup \partial D_4 \cup \R$ is the contour shown in Figure \ref{Djsreversed.pdf}.
The upshot of the above conjugation is that we can now factorize the jump across $\R$ as follows:
\begin{align}\label{tildeJonRIV}
  \tilde{J} = 
 B_u^{-1} B_l,   \qquad k \in \R, 
\end{align}
where
\begin{align}\nonumber
 B_l =  \begin{pmatrix} 1 & 0	\\
  i\delta_+(k)^{-2} r_1(k) e^{t\Phi(\zeta,k)}	& 1 \end{pmatrix}, 
	\qquad
B_u =  \begin{pmatrix} 1	& -i\delta_-(k)^2 r_4(k) e^{-t\Phi(\zeta,k)}	\\
0	& 1 \end{pmatrix},
\end{align}
and the functions $r_1(k)$ and $r_4(k)$ are defined by
\begin{align*}
 r_1(k) = \frac{r(k)}{1 - r(k)\overline{r(\bar{k})}}, \qquad
r_4(k) = \frac{\overline{r(\bar{k})}}{1 - r(k)\overline{r(\bar{k})}}.
\end{align*}
The factors $e^{\pm \frac{i\pi}{4}\sigma_3}$ in (\ref{MtildedefIV}) are inserted for later convenience (more precisely, to ensure that (\ref{RinequalitiesIV}) holds with $s$ pure imaginary).

\medskip
{\bf Step 4: Introduce analytic approximations of $r_1(k)$ and $r_4(k)$.}
Lemma \ref{decompositionlemma2} applies also for $(x,t) \in \mathbb{S}$. Thus we find decompositions 
$$r_j(k) = r_{j,a}(x, t, k) + r_{j,r}(x, t, k), \qquad j = 1,4, \quad k \in \R,$$
where $r_{j, a}$ and $r_{j,r}$ have the properties listed in Lemma \ref{decompositionlemma2}.

\begin{figure}
\begin{center}
\medskip
 \begin{overpic}[width=.50\textwidth]{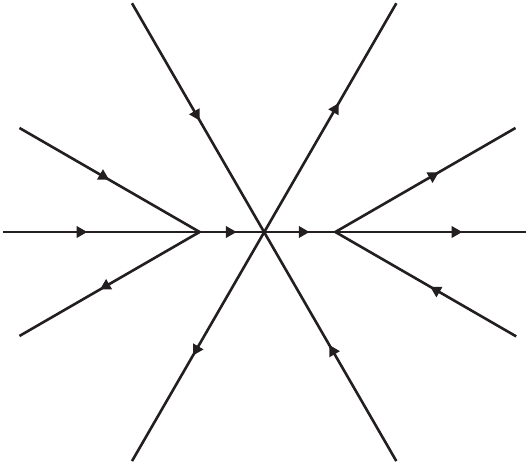}
      \put(62,39){$k_0$}
      \put(35,39){$-k_0$}
      \put(78.5,57.2){$Y_1$}
      \put(19,57){$Y_1$}
      \put(19,28){$Y_2$}
      \put(78.5,28){$Y_2$}
      \put(90,50){$V_1$}
      \put(90,35){$V_2$}
      \put(8,50){$V_1$}
      \put(8,35){$V_2$}
       \end{overpic}
    \qquad \qquad
     \begin{figuretext}\label{GammaIV.pdf}
       The  contour  $\Gamma = \Sigma \cup Y_1 \cup Y_2$ and the open sets $V_1$ and $V_2$.
     \end{figuretext}
     \end{center}
\end{figure}

\medskip
{\bf Step 5: Deform again.}
Let $\Gamma = \Sigma \cup Y_1 \cup Y_2$ and, for $j = 1,2$, let $V_j$ denote the open region between $\R$ and $Y_j$, see Figure \ref{GammaIV.pdf}. Write
$$B_l = B_{l,r}B_{l,a}, \qquad B_u = B_{u,r}B_{u,a},$$
where $\{B_{l,a}, B_{u,a}\}$ and $\{B_{l,r}, B_{u,r}\}$ denote the matrices $\{B_{l}, B_{u}\}$ with $\{r_1(k), r_4(k)\}$ replaced with $\{r_{1,a}(k), r_{4,a}(k)\}$ and $\{r_{1,r}(k), r_{4,r}(k)\}$, respectively.
The estimate (\ref{rjaestimateb}) implies that 
\begin{align*}
\begin{cases}
B_{l,a}(x,t,\cdot), B_{l,a}^{-1}(x,t,\cdot) \in I + (\dot{E}^2 \cap E^\infty)(V_1),
	\\
B_{u,a}(x,t,\cdot), B_{u,a}^{-1}(x,t,\cdot) \in I + (\dot{E}^2 \cap E^\infty)(V_2),
\end{cases}
\end{align*}
for each $(x,t) \in \mathbb{S}$.
Hence $\tilde{M}$ satisfies the $L^2$-RH problem (\ref{RHMtildeIV}) iff the function $m(x,t,k)$ defined by
\begin{align*}
m(x,t,k) = \begin{cases}  
\tilde{M}(x,t,k) B_{l,a}(x,t,k)^{-1}, & k \in V_1, \\
\tilde{M}(x,t,k) B_{u,a}(x,t,k)^{-1}, & k \in V_2, \\
\tilde{M}(x,t,k), & \text{elsewhere},
\end{cases}
\end{align*}
satisfies the $L^2$-RH problem
\begin{align}\label{RHm2IV}
\begin{cases} m(x, t, \cdot) \in I + \dot{E}^2(\C \setminus \Gamma), \\
m_+(x, t, k) = m_-(x, t, k) v(x, t, k) \quad \text{for a.e.} \ k \in \Gamma, 
\end{cases} 
\end{align}
where, in view of (\ref{tildeJdef}) and (\ref{tildeJonR}), the jump matrix $v$ is given by
\begin{align}\label{vdefIV}
v = \begin{cases}
B_{l,a} = \begin{pmatrix} 1 & 0	\\
 i \delta(k)^{-2} r_{1,a}(x,t,k) e^{t\Phi(\zeta,k)}	& 1 \end{pmatrix}, & k \in Y_1, 
 	\\
B_{u,a} = \begin{pmatrix} 1	& - i \delta(k)^2 r_{4,a}(x,t,k) e^{-t\Phi(\zeta,k)}	\\
0	& 1 \end{pmatrix}, & k \in Y_2.
	\\
B_{u}^{-1} B_{l}, & k \in Y_3, 
	\\
B_{u,r}^{-1} B_{l,r}, & k \in \R \setminus Y_3, 
 	\\
\begin{pmatrix} 1 & 0 \\ i \delta(k)^{-2}  h_{r}(t, k) e^{t\Phi(\zeta, k)} & 1 \end{pmatrix}, & k \in \partial D_1,
 	\\
\begin{pmatrix} 1 & -i \delta(k)^2 \overline{h_r(t, \bar{k})} e^{-t\Phi(\zeta, k)} \\0 & 1 \end{pmatrix}, & k \in \partial D_4.
\end{cases}
\end{align}

\medskip
{\bf Step 6: Apply Theorem \ref{steepestdescentthIV}.}
We verify that Theorem \ref{steepestdescentthIV} can be applied to the contour $\Gamma$ and the jump matrix $v$ with $s = i \delta_+(0)^{-2} r_{1}(0)$ and
\begin{align*}
& R_1(x,t,k) =  i\delta(k)^{-2} r_{1,a}(x,t,k), &&
R_2(x,t,k) = -i\delta(k)^2 r_{4,a}(x,t,k), 
	\\
& R_3(x,t,k) = i \delta_+(k)^{-2} r_1(k), &&
R_4(x,t,k) = -i\delta_-(k)^2 r_4(k).
\end{align*}

\begin{figure}
\begin{center}
\medskip
 \begin{overpic}[width=.45\textwidth]{GammaIV.pdf}
     \put(48,66.5){$\Gamma$}
           \put(62,40){\tiny $k_0$}
      \put(34.5,40){\tiny $-k_0$}
   \end{overpic}
       \qquad 
 \begin{overpic}[width=.452\textwidth]{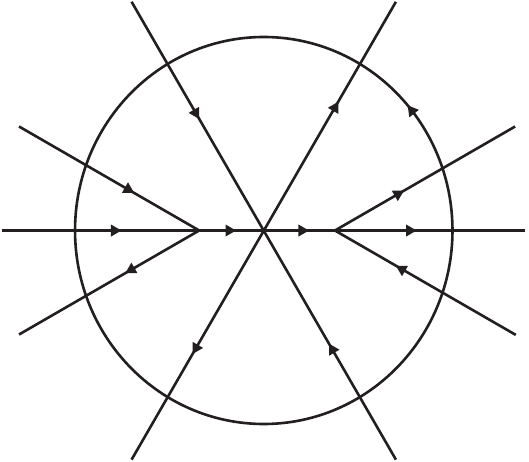}
      \put(48,67){$\hat{\Gamma}$}
          \put(8,40){\tiny $-1$}
          \put(86.5,40){\tiny $1$}
          \put(62,40){\tiny $k_0$}
      \put(34.5,40){\tiny $-k_0$}
     \end{overpic}
     \begin{figuretext}\label{GammaIVhat.pdf}
       The contours $\Gamma$ and $\hat{\Gamma}$ satisfy the properties {\upshape ($\Gamma$1)-($\Gamma$4)} of Theorem \ref{steepestdescentthIV}.
     \end{figuretext}
     \end{center}
\end{figure}

The contours $\Gamma$ and $\hat{\Gamma}$ are shown in Figure \ref{GammaIVhat.pdf}.
The conditions ($\Gamma$1)-($\Gamma$4) are clearly satisfied. 
Since the contour $k_0^{-1}\Gamma$ is independent of $(x,t)$, a scaling argument shows that $\|\mathcal{S}_{\Gamma}\|_{\mathcal{B}(L^2(\hat{\Gamma}))}$ is independent of $(x,t) \in \mathbb{S}$. Moreover, the Cauchy singular operator $\mathcal{S}_{\{|k| = 1\}}$ associated with the unit circle is bounded. Hence a decomposition argument shows that $\mathcal{S}_{\hat{\Gamma}}$ is uniformly bounded on $L^2(\hat{\Gamma})$ (cf. the proof of Theorem 4.23 in \cite{BK1997}).
Equation (\ref{winL1L2LinfIV}) follows from (\ref{vdefIV}) and the estimates in Lemmas \ref{decompositionlemma} and \ref{decompositionlemma2}.
Clearly $\det v = 1$. The symmetry condition (\ref{vsymmIV}) follows from (\ref{vdefIV}) and the  symmetries of $h_r$, $r_{j,a}$, $r_{j,r}$, and $\delta$.

We next verify (\ref{wL12inftyIV}).
By part $(d)$ of Lemma \ref{decompositionlemma2}, the $L^1$ and $L^\infty$ norms of $w$ are $O(t^{-\frac{3}{2}})$ on $\R \setminus Y_3$ uniformly with respect to $(x,t) \in \mathbb{S}$.
By part $(c)$ of Lemma \ref{decompositionlemma}, the $L^1$ and $L^\infty$ norms of $w$ are $O(t^{-\frac{3}{2}})$ on $\partial D_1 \cup \partial D_4$ uniformly with respect to $(x,t) \in \mathbb{S}$.
This verifies (\ref{wL12inftyIV}).

It remains to verify (\ref{RinequalitiesIV}). 
Since $r(k) = \overline{r(-\bar{k})}$, we have $r(0) \in \R$. Moreover, the symmetries $\delta(k) = \overline{\delta(-\bar{k})} =  \delta(-k)^{-1}$ imply $\delta_+(0) = \overline{\delta_+(0)} = \delta_-(0)^{-1}$. In particular, $s \in i \R$. It also follows that
$$R_j(x,t,0) = \begin{cases} i \delta_+(0)^{-2} \frac{r(0)}{1 - |r(0)|^2} = s, & j = 3, \vspace{.12cm}\\ 
-i\delta_-(0)^2 \frac{\overline{r(0)}}{1 - |r(0)|^2} = -s, \quad & j = 4.
\end{cases}$$
To see that $|s| < 1$, we write $\delta(k) = e^{\chi(k)}$, where $\chi(k) = - \frac{1}{2\pi i}\int_\R \ln(1 - |r(s)|^2) \frac{ds}{s - k}$.
The Sokhotski-Plemelj theorem implies
\begin{align}\label{chipluszero}
\chi_+(0) = - \frac{1}{2\pi i}\dashint_\R \ln\big(1 - |r(s)|^2\big) \frac{ds}{s} - \frac{1}{2}\ln(1 - |r(0)|^2).
\end{align}
Hence
$$|\delta_+(0)|^{-2} = e^{-2\re \chi_+(0)} = 1 - |r(0)|^2,$$
which implies $|s| = |r(0)|$. The assumption $\sup_{k \in \R} |r(k)| < 1$ now gives $|s| < 1$.  

Finally, we prove (\ref{RinequalitiesIV}) for $R_1$ and $R_3$; the proofs for $R_2$ and $R_4$ are similar. 
Using the expression (\ref{chipluszero}) for $\chi_+(0)$, we obtain, for $\im k > 0$,
\begin{align*}
\chi(k) - \chi_+(0) = & - \frac{1}{2\pi i} \dashint_\R \ln\big(1 - |r(s)|^2\big) \bigg(\frac{1}{s - k} - \frac{1}{s}\bigg) ds + \frac{1}{2}\ln(1 - |r(0)|^2)
	\\
= & - \frac{k}{2\pi i} \dashint_\R \frac{1}{s}\ln\bigg(\frac{1 - |r(s)|^2}{1 - |r(0)|^2}\bigg) \frac{ds}{s-k} + \bigg[\frac{1}{2} - \frac{k}{2\pi i} \dashint_\R \frac{ds}{s(s-k)}\bigg]\ln(1 - |r(0)|^2).
\end{align*}
The expression within square brackets vanishes. Hence
$$\chi(k) - \chi_+(0) = -\frac{k}{2\pi i} \int_\R \frac{1}{s}\ln\bigg(\frac{1 - |r(s)|^2}{1 - |r(0)|^2}\bigg) \frac{ds}{s-k}, \qquad \im k > 0.$$
The function $\frac{1}{s}\ln(\frac{1 - |r(s)|^2}{1 - |r(0)|^2})$ belongs to $H^1(\R)$; thus (cf. (\ref{chibound}))
$$|\chi(k) - \chi_+(0)| \leq C |k|, \qquad \im k > 0.$$
The inequality (\ref{ewminus1estimate}) now yields
\begin{align}\label{chichiIV}
|1 - e^{2(\chi(k) - \chi_+(0))}| \leq 2|\chi(k) - \chi_+(0)| e^{2|\chi(k) - \chi_+(0)|}
\leq C |\chi(k) - \chi_+(0)| \leq C |k|
\end{align}
whenever $\im k > 0$. On the other hand, by (\ref{rjaestimatea}), (\ref{rjaestimateb}), and Lemma \ref{DeltalemmaIV},
\begin{align}\nonumber
|R_1(x,t,k) - s|  \leq &\; 
|\delta(k)|^{-2} |r_{1,a}(x,t,k) - r_1(k_0)|
+ |\delta(k)|^{-2} |r_1(k_0) - r_1(0)|
	\\\nonumber
& + |\delta(k)^{-2} - \delta_+(0)^{-2}| |r_1(0)|
	\\ \label{R1sIV}
 \leq &\; C |k - k_0| e^{\frac{t}{4}|\re \Phi(\zeta,k)|} + C|k_0|
+ C |1 - e^{2(\chi(k) - \chi_+(0))}|, \qquad k \in Y_1^+,
\end{align}
where we recall that $Y_1^+$ and $Y_1^-$ denote the parts of $Y_1$ that lie in the right and left half-planes, respectively. Equations (\ref{chichiIV}) and (\ref{R1sIV}) prove (\ref{RinequalitiesIV}) for $R_1(x,t,k)$ with $k \in Y_1^+$; the proof when $k \in Y_1^-$ is analogous. 
The estimate (\ref{chichiIV}) also implies
\begin{align*}
|R_3(x,t,k) - s| & \leq |\delta_+(k)|^{-2} |r_1(k) - r_1(0)| + |\delta_+(k)^{-2} - \delta_+(0)^{-2}| |r_1(0)| 
	\\
& \leq  C|k| + C |1 - e^{2(\chi_+(k) - \chi_+(0))}| 
\leq  C|k|, \qquad k \in Y_3,
\end{align*}
which proves (\ref{RinequalitiesIV}) also for $R_3$.

\medskip
{\bf Step 7: Find asymptotics.}
As a consequence of Theorem \ref{steepestdescentthIV}, the $L^2$-RH problem (\ref{RHm2IV}) for $m$ has a unique solution and the limit in (\ref{limlm12IV}) exists for all $(x,t) \in \mathbb{S}$ with $t$ sufficiently large. Since the RH problems for $M$ and $m$ are equivalent, the $L^2$-RH problem (\ref{RHM}) for $M$ also has a unique solution for all sufficiently large $t$. Moreover, the limit (\ref{ulim}) defining $u(x,t)$ exists whenever $t$ is large enough and
\begin{align*}
u(x,t) & = -2i\ntlim_{k\to\infty} (kM(x,t,k))_{12} = 2 \ntlim_{k\to\infty} (km(x,t,k))_{12}.
\end{align*}

Equation (\ref{limlm12IV}) of Theorem \ref{steepestdescentthIV} then yields
\begin{align*}
u(x,t) & = \frac{u^P\big(\frac{-x}{(3t)^{1/3}}; s, 0,-s\big)}{(3t)^{\frac{1}{3}}}  +  O\big(t^{-\frac{2}{3}}\big), \qquad t \to \infty, \quad (x,t) \in \mathbb{S},
\end{align*}
where the error term is uniform with respect to $x$.
This proves (\ref{ufinalIV}) and completes the proof of Theorem \ref{mainth2}.
\proofend

\section{Concluding remarks}\nequation\label{concludesec}
Good asymptotic formulas provide one of the most efficient ways to obtain both quantitative and qualitative information on the solution of a PDE. Since the solution of an integrable equation can be expressed in terms of the solution of a matrix Riemann-Hilbert problem, the nonlinear steepest descent method introduced by Deift and Zhou \cite{DZ1993} provides an excellent tool for establishing such formulas. For initial value problems, this method has been utilized extensively since its introduction in 1993. 

In this paper, we have considered the derivation of asymptotic formulas for initial-{\it{}boundary} value problems via the nonlinear steepest descent method. In particular, we have established asymptotic formulas valid in the similarity and self-similar sectors for the mKdV equation in the quarter plane. A special emphasis has been placed on giving complete proofs and on the derivation of rigorous and uniform error estimates. 

Previous studies addressing the question of asymptotics for initial-boundary value problems for integrable equations include \cite{BFS2004, BS2009, FI1994, FI1996, FIS2005}. In particular, it was shown in \cite{BFS2004} that the quarter plane solution $u(x,t)$ of the mKdV equation (\ref{mkdv}) satisfies $u(x,t) = O(t^{-1/2})$ as $(x,t) \to \infty$ in a subsector of the similarity sector. Our results give the explicit form of the leading order term in this formula. We have also demonstrated that the boundary values have a large effect on the asymptotic behavior of the solution in the self-similar region. In particular, if the boundary values have sufficient decay as $t \to \infty$, then we have showed that the solution is uniformly of $O(t^{-2/3})$ throughout the self-similar sector. In contrast, in the absence of a boundary, the solution has a leading order term of order $O(t^{-1/3})$ in this sector. 

\appendix
\section{The RH problem associated with Painlev\'e II} \label{painleveapp}
\renewcommand{\theequation}{A.\arabic{equation}}\nequation
We give a number of results related to the Painlev\'e II equation (\ref{painleve2}). We refer to \cite{FIKN2006} for full proofs. Equation (\ref{painleve2}) is the compatibility condition 
$$A_y - U_z + [A, U] = 0$$
of the Lax pair
\begin{align}\label{painlevelax}
\begin{cases}
  \Psi_z = A \Psi, \\
  \Psi_y = U \Psi,
  \end{cases}
\end{align}
where
\begin{align*}
 & A(y, z) = -i(4z^2 + y + 2u^2)\sigma_3 - 4z u \sigma_2 - 2u_y \sigma_1,
  	\\
 & U(y, z) = -iz \sigma_3 - u\sigma_2, 
\end{align*}  
and $\{\sigma_j\}_1^3$ denote the standard Pauli matrices. The linear system (\ref{painlevelax}) has an irregular singular point at $z = \infty$.
The theory of ODEs with singular points implies that there exist solutions $\{\Psi_n\}_1^6$ of (\ref{painlevelax}) such that 
\begin{align}\label{Psinasymptotics}
\Psi_n(y, z) = \bigl(I + O(z^{-1})\bigr) e^{-i(yz + \frac{4}{3}z^3)\sigma_3}, \qquad z \to \infty, \quad z \in \Omega_n,
\end{align}
where the sectors $\Omega_n$ are given by
$$\Omega_n = \biggl\{ z \in \C \,\bigg|\, \arg z \in \biggl(\frac{\pi}{3}(n-2), \frac{\pi}{3}n\biggr)\biggr\}, \qquad n = 1, \dots 6.$$
The asymptotics in (\ref{Psinasymptotics}) is uniform as $z \to \infty$ within any closed subwedge of $\Omega_n$.
The $\Psi_n$'s are entire functions of $z$ which are related by
$$\Psi_{n+1} = \Psi_n S_n, \qquad n = 1, \dots, 6, \quad \Psi_7:= \Psi_1,$$
where the constant Stokes matrices $\{S_n\}_1^6$ have the triangular form
$$S_n = \begin{pmatrix} 1 & 0 \\ s_n & 1 \end{pmatrix}, \quad n \text{ odd};
\qquad
S_n = \begin{pmatrix} 1 & s_n \\ 0 & 1 \end{pmatrix}, \quad n \text{ even}.$$
The symmetries
$$A(y, z ) = -\sigma_2 A(y, -z) \sigma_2, \qquad
U(y, z ) = \sigma_2 U(y, -z) \sigma_2$$
imply that
$$S_{n+3} = \sigma_2 S_n \sigma_2 \quad \text{i.e.} \quad
s_{n+3} = - s_n, \qquad n = 1,2,3.$$

\begin{figure}
\begin{center}
\bigskip
 \begin{overpic}[width=.45\textwidth]{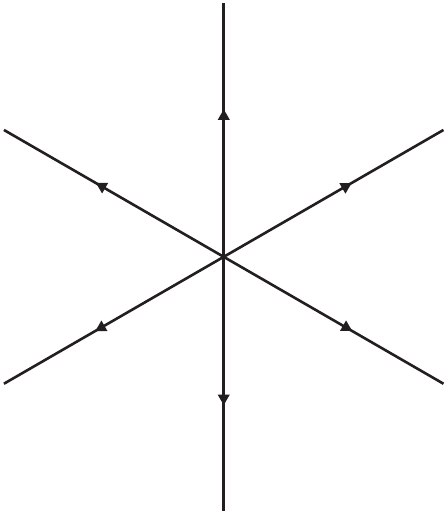}
 \put(67,57){$\Gamma_1$}
 \put(46,76){$\Gamma_2$}
 \put(19,67){$\Gamma_3$}
 \put(15,40){$\Gamma_4$}
 \put(36,20){$\Gamma_5$}
 \put(64,29){$\Gamma_6$}
 \end{overpic}
   \begin{figuretext}\label{Sixrays.pdf}
      The contour $\Gamma = \cup_{j=1}^6 \Gamma_j$.
      \end{figuretext}
   \end{center}
\end{figure}
Let $\Gamma = \cup_{j=1}^6 \Gamma_j$ where the rays $\{\Gamma_j\}_1^6$ are defined by (see Figure \ref{Sixrays.pdf})
$$\Gamma_j = \biggl\{z \in \C \, \bigg|\, \arg z = \frac{\pi}{6} + \frac{\pi (n-1)}{3}\biggr\}, \qquad n = 1, \dots, 6.$$
Assume that these rays are oriented away from the origin.
Then the sectionally analytic function $m^P(y,z)$ defined by
$$m^P(y,z) = \Psi_n(y,z) e^{i(yz + \frac{4}{3}z^3)\sigma_3}, \qquad \arg z \in \biggl(\frac{\pi(n-2)}{6}, \frac{\pi n}{6}\biggr),$$
satisfies the RH problem
\begin{align}\label{RHmP}
\begin{cases} m_+^P(y, z) =  m_-^P(y, z) e^{-i(yz + \frac{4}{3}z^3)\hat{\sigma}_3} S_n, & z \in \Gamma_n, \quad n = 1, \dots, 6, \\
m^P(y, z) = I + O(z^{-1}), & z \to \infty.
\end{cases} 
\end{align}
The cyclicity condition $S_1S_2S_3S_4S_5S_6 = I$ reduces to the condition
\begin{align}\label{s123condition}
s_1 -s_2 +s_3 + s_1s_2s_3 = 0.
\end{align}

It can be shown (see Theorem 3.4 and Theorem 4.2 in \cite{FIKN2006}) that given a set of complex constants $\mathcal{S} = \{s_1, s_2, s_3\}$ satisfying (\ref{s123condition}), there exists a countable set $Y_{\mathcal{S}} = \{y_j = y_j(\mathcal{S})\}_{j=1}^\infty \subset \C$ with $y_j \to \infty$ as $j \to \infty$, such that the RH problem (\ref{RHmP}) has a unique classical solution\footnote{A classical solution of (\ref{RHmP}) is a $2\times 2$-matrix valued function $m^P(y, \cdot)$  such that $(i)$ for each component $D$ of $\C \setminus \Gamma$, the restriction of $m^P$ to $D$ is analytic in $D$ and has a continuous extension to $\bar{D}$, $(ii)$ $m^P$ satisfies the jump condition in (\ref{RHmP}), and $(iii)$ $m^P(y, z) = I + O(z^{-1})$ uniformly as $z \to \infty$, $z \in \C$.} $m^P(y, z)$ for each $y \in \C \setminus Y_{\mathcal{S}}$. 
The functions $m_n^P(y, z)$, where $m_n^P$ denotes the restriction of $m^P$ to $\arg z \in (\frac{\pi(2n-3)}{6}, \frac{\pi (2n-1)}{6})$, admit analytic continuations to $(\C \setminus Y_{\mathcal{S}}) \times \C$.
As $z \to \infty$, $m^P$ satisfies 
\begin{align}\label{mPasymptotics}
m^P(y, z) = I + \frac{1}{2z} \begin{pmatrix} 0 & u^P(y; s_1, s_2, s_3) \\ u^P(y; s_1, s_2, s_3) & 0 \end{pmatrix} + O(z^{-2}), \qquad z \to \infty, \quad y \notin Y_{\mathcal{S}},
\end{align}
where $u^P(\cdot; s_1, s_2, s_3)$ satisfies the Painlev\'e II equation (\ref{painleve2}). 
The asymptotics in (\ref{mPasymptotics}) is uniform in the following sense: For any compact subset $K \subset \C \setminus Y_{\mathcal{S}}$, there exists a $C > 0$ such that
$$\bigg\|m^P(y,z) - I - \frac{1}{2z} \begin{pmatrix} 0 & u^P(y; s_1, s_2, s_3) \\ u^P(y; s_1, s_2, s_3) & 0 \end{pmatrix} \bigg\| < C |z|^{-2}, \qquad |z| > 1, \quad y \in K.$$
Moreover, the map $\mathcal{S} \mapsto u^P(\cdot; s_1, s_2, s_3)$ is a bijection
\begin{align}\label{painlevebijection}
\{(s_1,s_2,s_3) \in \C^3 \,|\, s_1 -s_2 +s_3 + s_1s_2s_3 = 0\} \to \{\text{solutions of }(\ref{painleve2})\}
\end{align}
and the set $Y_{\mathcal{S}}$ is the set of poles of $u^P(\cdot; s_1, s_2, s_3)$, see Corollary 4.4 in \cite{FIKN2006}.

We will be interested in solutions $u$ of the Painlev\'e II equation which are real-valued on $\R$. To this end, we note that (p. 158 in \cite{FIKN2006})
$$u^P(y; s_1, s_2, s_3) = \overline{u^P(\bar{y}; s_1, s_2, s_3)}$$ 
if and only if the $s_n$'s satisfy
$$s_3 = \bar{s}_1, \qquad s_2 = \bar{s}_2.$$

\bigskip
\noindent
{\bf Acknowledgement} {\it The author thanks the two referees for excellent suggestions. Support is acknowledged from the Swedish Research Council Grant No. 2015-05430, the G\"oran Gustafsson Foundation, Sweden, and the EPSRC, UK.}

\bibliographystyle{plain}
\bibliography{is}

\begin{thebibliography}{99}
\small

\bibitem{AK2015}
D. C. Antonopoulou and S. Kamvissis,
On the Dirichlet to Neumann problem for the $1$-dimensional cubic NLS equation on the half-line,
{\it Nonlinearity} {\bf 28} (2015), 3073--3099.


\bibitem{BDT1988}
R. Beals, P. Deift, and C. Tomei, {\it Direct and inverse scattering on the line}, Mathematical Surveys and Monographs 28, American Mathematical Society, Providence, RI, 1988.

\bibitem{BFS2004}
A. Boutet de Monvel, A. S. Fokas, and D. Shepelsky, The mKdV equation on the half-line,
{\it J. Inst. Math. Jussieu} {\bf 3} (2004), 139--164. 

\bibitem{BIK2009}
A. Boutet de Monvel, A. Its, and V. Kotlyarov, Long-time asymptotics for the focusing NLS equation with time-periodic boundary condition on the half-line, {\it Comm. Math. Phys.} {\bf 290} (2009), 479--522.

\bibitem{BK2003}
A. Boutet de Monvel and V. Kotlyarov, Generation of asymptotic solitons of the nonlinear Schršdinger equation by boundary data, {\it J. Math. Phys.} {\bf 44} (2003), 3185--3215. 

\bibitem{BK2007}
A. Boutet de Monvel and V. Kotlyarov, The focusing nonlinear Schr\"odinger equation on the quarter plane with time-periodic boundary condition: a Riemann-Hilbert approach, {\it J. Inst. Math. Jussieu} {\bf 6} (2007), 579--611. 

\bibitem{BKS2009}
A. Boutet de Monvel, V. Kotlyarov, and D. Shepelsky, Decaying long-time asymptotics for the focusing NLS equation with periodic boundary condition, {\it Int. Math. Res. Not. IMRN} {\bf 2009}, 547--577.

\bibitem{BS2009}
A. Boutet de Monvel and D. Shepelsky, Long time asymptotics of the Camassa-Holm equation on the half-line, {\it Ann. Inst. Fourier (Grenoble)} {\bf 59} (2009), 3015--3056.

\bibitem{BK1997}
A. B\"ottcher and Y. I. Karlovich, {\it Carleson curves, Muckenhoupt weights, and Toeplitz operators}, Progress in Mathematics, 154. Birkh\"auser Verlag, Basel, 1997.
 
\bibitem{DZ1993}
P. Deift and X. Zhou, A steepest descent method for oscillatory Riemann-
Hilbert problems. Asymptotics for the MKdV equation, 
{\it Ann. of Math.} {\bf 137} (1993), 295--368.

\bibitem{FI1994}
A. S. Fokas and A. R. Its, An initial-boundary value problem for the Korteweg-de Vries equation, 
{\it Math. Comput. Simulation} {\bf 37} (1994), 293--321.
 
\bibitem{FI1996}
A. S. Fokas and A. R. Its, The linearization of the initial-boundary value problem of the nonlinear Schr\"odinger equation,
{\it SIAM J. Math. Anal.} {\bf 27} (1996), 738--764. 


\bibitem{FIKN2006}
A. S. Fokas, A. R. Its, A. A. Kapaev, V. Y. Novokshenov, Painlev\'e transcendents. The Riemann-Hilbert approach. Mathematical Surveys and Monographs, 128. American Mathematical Society, Providence, RI, 2006.

\bibitem{FIS2005}
A. S. Fokas, A. R. Its, and L.-Y. Sung, The nonlinear Schr\"odinger equation on the half-line, 
{\it Nonlinearity} {\bf 18} (2005), 1771--1822.

\bibitem{HM1980}
S. P. Hastings and J. B. McLeod, A boundary value problem associated with the second Painlev\'e transcendent and the Korteweg-de Vries equation, {\it Arch. Rational Mech. Anal.} {\bf 73} (1980), 31--51. 

\bibitem{IN1986}
A. R. Its and V. Y. Novokshenov, {\it The isomonodromic deformation method in the theory of Painlev\'e equations}, Lecture Notes in Mathematics 1191, Springer-Verlag, Berlin, 1986.

\bibitem{Lmkdvrigorous}
J. Lenells, Nonlinear Fourier transforms and the mKdV equation in the quarter plane, {\it Stud. Appl. Math.} {\bf 136} (2016), 3--63.

\bibitem{Lnonlinearsteepest}
J. Lenells, The nonlinear steepest descent method for Riemann-Hilbert problems of low regularity, preprint, arXiv:1501.05329.

\end{thebibliography}

\end{document}